\documentclass[pdflatex,sn-mathphys-num]{sn-jnl}

\usepackage{amsmath,dsfont, amssymb}
\usepackage{mathtools}
\usepackage{amssymb}
\usepackage{amsthm}
\usepackage{graphicx}
\usepackage{booktabs}
\usepackage{multirow}
\usepackage{adjustbox}
\usepackage{float}
\usepackage{amsfonts}
\usepackage{todonotes}
\usepackage{mathrsfs}
\usepackage{algpseudocode}
\usepackage{subcaption}
\usepackage{placeins}
\usepackage{enumitem}

\theoremstyle{thmstyleone}%
\newtheorem{theorem}{Theorem}
\newtheorem{proposition}[theorem]{Proposition}
\newtheorem{lemma}[theorem]{Lemma}
\newtheorem{Algorithm}[theorem]{Algorithm}
\newtheorem{assumption}[theorem]{Assumption}

\theoremstyle{thmstyletwo}%
\newtheorem{example}{Example}%
\newtheorem{remark}{Remark}%

\theoremstyle{thmstylethree}%
\newtheorem{definition}{Definition}%

\DeclareMathOperator*{\grad}{grad}

\DeclareMathOperator*{\amin}{argmin}
\DeclareMathOperator*{\prox}{prox}
\DeclareMathOperator*{\ep}{EP}

\newcommand{\R}{\mathbb{R}}
\newcommand{\M}{\mathbb{M}}
\newcommand{\N}{\mathbb{N}}

\begin{document}

\title[Regularized Extragradient Methods for Solving Equilibrium Problems on Hadamard Manifolds]{Regularized Extragradient Methods for Solving Equilibrium Problems on Hadamard Manifolds}


\author[1,2]{\fnm{Shikher} \sur{Sharma}}\email{shikhers043@gmail.com}

\author*[2]{\fnm{Pankaj} \sur{Gautam}}\email{pgautam908@gmail.com}
\equalcont{These authors contributed equally to this work.}

\author[1]{\fnm{Simeon} \sur{Reich}}\email{sreich@technion.ac.il}
\equalcont{These authors contributed equally to this work.}

\affil[1]{\orgdiv{Department of Mathematics}, \orgname{The Technion -- Israel Institute of Technology},  \city{Haifa}, \postcode{32000}, \country{Israel}}

\affil[2]{\orgdiv{Department of Applied Mathematics and Scientific Computing}, \orgname{IIT Roorkee},  \country{India}}


	\abstract{Employing two distinct types of regularization terms, we propose two regularized extragradient methods for solving equilibrium problems on Hadamard manifolds. The sequences generated by these extragradient algorithms converge to a solution of the equilibrium problem without requiring the Lipschitz continuity of the bifunction {}{and without an upper bound on the regularization parameters}. We establish convergence results for both algorithms under a monotonicity condition. When the bifunction is {}{additionally strongly pseudomonotone}, we derive global error bounds and establish $R$-linear convergence rates. Finally, we present numerical experiments to demonstrate the effectiveness of our methods.
	}

\maketitle

\section{Introduction}

The equilibrium problem (EP) has been extensively studied and remains a highly active area of research. One of the key motivations behind its investigation is its broad applicability, since numerous mathematical problems can be formulated as equilibrium problems. These include optimization problems, Nash equilibrium problems, complementarity problems, fixed point problems, and variational inequality problems. A comprehensive discussion of EP and its developments can be found in the works of Blum and Oettli \cite{Blum1994}, Bianchi and Schaible \cite{Bianchi1996}, and the references therein.  In this context, Moudafi \cite{Moudafi1999} extended the proximal point algorithm to equilibrium problems with monotone bifunctions, and later Konnov \cite{Konnov2003} proposed a variant under weaker assumptions. Generally, the method is formulated for monotone equilibrium problems, ensuring that each regularized subproblem is strongly monotone.

The study of the proximal point method on Hadamard manifolds has gained considerable attention, with various researchers exploring its application to specific cases of the equilibrium problem (EP); see, for instance, the works of Ferreira and Oliveira~\cite{Ferreira2002}, Li et al.~\cite{Martin2}, and Tang et al.~\cite{Tang2013}.
More recently, there has been a natural shift toward extending classical Euclidean concepts and methodologies to the Riemannian setting. Notable contributions in this direction include those by Ledyaev and Zhu \cite{Ledyaev2007}, Ferreira et al. \cite{Ferreira2005}, Li et al. \cite{Li2011SJO,Li2011SVVA}, among others. A significant advantage of adapting these techniques to Riemannian manifolds is the ability to reformulate certain nonconvex problems as convex ones by selecting an appropriate Riemannian metric. Relevant discussions on this topic can be found in the papers by Rapcs\'ak \cite{Rapscak}, Cruz Neto et al. \cite{Neto2002}, Bento and Melo \cite{Bento2012}, and by Colao et al. \cite{Colao2012}.

Note that Hadamard manifolds generally do not have a linear structure, which indicates that properties, techniques, and algorithms developed in linear spaces cannot generally be applied directly in Hadamard manifolds. Therefore, it is valuable and interesting to extend algorithms for solving equilibrium problems from linear spaces to Hadamard manifolds.

Let $\M$ be an Hadamard manifold. Let $C \subset \M$ be a nonempty, closed and geodesically convex subset, and let $F \colon C \times C \to \R$ be  a bifunction satisfying $F(x,x)=0$ for all $x\in C$. The equilibrium problem in the Riemannian context (denoted by $\ep(F,C)$) is formulated as follows:
\begin{equation}\label{equib-prob}
	\text{Find } x^*\in C \text{ such that } F(x^*,y)\geq 0 \text{ for all } y \in C.
\end{equation}
In recent years, various extragradient-based methods have been developed in order to solve equilibrium problem \eqref{equib-prob} on Hadamard manifolds. Colao et al. \cite{Colao2012} established the existence of solutions to such problems under suitable conditions on the bifunction. Building upon their work, Cruz Neto et al. \cite{Neto2016OL} proposed an extragradient algorithm with a fixed step size and proved its global convergence. Inspired by the works of Cruz Neto et al. \cite{Neto2016OL} and Hieu \cite{Hieu2018AA,Hieu2019}, Khammahawong et al. \cite{Khammahawong2020} introduced two extragradient methods employing decreasing and nonsummable step sizes, extending Hieu's results from Hilbert space to the Hadamard manifold setting. For strongly pseudomonotone bifunctions, they demonstrated convergence of the generated sequences under mild assumptions.

Further extending this line of research, Ansari and Islam \cite{Ansari2020} proposed two adaptive explicit extragradient algorithms for solving EP \eqref{equib-prob} on Hadamard manifolds. Their methods utilize a simple, information-based adaptive step size rule, eliminating the need for prior knowledge of the Lipschitz constant. Convergence was established under pseudomonotonicity, with linear convergence shown for strongly pseudomonotone bifunctions.

It is worth noting that the methods in \cite{Neto2016OL, Khammahawong2020, Ansari2020} require solving a strongly convex optimization problem twice per iteration. To reduce the computational cost, Chen et al.~\cite{Chen2021} proposed a modified golden ratio algorithm that requires only one evaluation per iteration.
Numerical experiments have shown that their method outperforms those of \cite{Neto2016OL, Khammahawong2020} in terms of efficiency and computational cost.

In 2022, Babu et al. \cite{Babu2022} developed an extragradient algorithm with an Armijo-type step size rule for solving equilibrium problems on Hadamard manifolds. Later, Oyewole and Reich \cite{Oyewole2023inertial} introduced an inertial self-adaptive variant. To reduce parameter dependence, Tan et al. \cite{Tan2024} proposed three adaptive extragradient methods that do not require prior knowledge of Lipschitz constants and compared them with the methods of Khammahawong et al. \cite{Khammahawong2020}. However, these methods remain computationally expensive due to the need to solve strongly convex subproblems multiple times per iteration.

In 2022, Bento et al. \cite{Bento2022AOR} introduced a regularization of the proximal point algorithm in the context of Hadamard manifolds using a new regularization term which has a clear interpretation in a recent variational rationality approach to human behavior. Later, Bento et al. \cite{Bento2022} introduced a new resolvent for the bifunction in terms of Busemann functions. An advantage of this new proposal is that, in addition to being a natural extension of its counterpart in the linear setting introduced by Combettes and Hirstoaga \cite{Hirstoaga}, the new regularization term is geodesically convex on general Hadamard manifolds. Bento et al. \cite{Bento2024} have recently considered the resolvent via Busemann functions introduced by Bento et al. \cite{Bento2022} and presented a proximal point method for solving equilibrium problems on Hadamard manifolds.

	Motivated by the works of Bento et al. \cite{Bento2022AOR, Bento2024} and inspired by the extragradient-type methods of Cruz Neto et al. \cite{Neto2016OL}, Babu et al. \cite{Babu2022}, and Tan et al. \cite{Tan2024}, we propose two regularized extragradient methods with distinct regularization schemes for solving equilibrium problems on Hadamard manifolds. In particular, the first algorithm employs a Busemann-function-based regularization term following Bento et al. \cite{Bento2024}, while the second incorporates a distance-type metric regularization inspired by Bento et al. \cite{Bento2022AOR}. The main contributions and novelty of this work are summarized as follows:

\begin{itemize}
		\item[(i)] The proposed methods incorporate Busemann-function-based and distance-type regularization terms, respectively. Both regularization terms are geodesically convex on Hadamard manifolds and support the well-posedness analysis of the regularized subproblems.
\item[(ii)] The algorithms operate {}{without upper step size restrictions} or line search procedures and do not require Lipschitz continuity of the bifunction, thereby extending applicability to monotone equilibrium problems in a more general setting.
\item[(iii)]   Global convergence is established under {}{a monotonicity condition}. In addition, global error bounds are derived, and $R$-linear convergence is proved under strong pseudomonotonicity, providing strong theoretical guarantees.
\item[(iv)]  Numerical experiments demonstrate that the proposed methods are efficient and {}{competitive on the tested examples} compared with recent methods in the literature.
\end{itemize}


\section{Preliminaries}

For the basic definitions and results regarding the geometry of Riemannian manifolds, we refer to \cite{Docarmo,Sakai,Udriste}. Let $\M$ be an $m$-dimensional differentiable manifold and let  $x \in \M$. The set of all tangent vectors at the point $x$, denoted by $T_x\M$, is called the tangent space of $\M$ at $x$. It forms a real vector space of dimension $m$. The tangent bundle of $\M$ is defined as the disjoint union $T\M = \bigcup_{x \in \M} T_x\M$.

We suppose that $\M$ is equipped with a Riemannian metric, that is,  a smoothly varying family of inner products on the tangent spaces $T_x\M$.  The metric at $x\in \M$ is denoted by $\langle \cdot, \cdot \rangle_x \colon T_x\M \times T_x\M \to \R$.
A differentiable manifold $\M$ with a Riemannian metric $\langle \cdot , \cdot \rangle$ is said to be a Riemannian manifold.
The norm corresponding to the inner product $\langle \cdot , \cdot \rangle_x$ on $T_x\M$ is denoted by $\|\cdot \|_x$.

Let $x,y \in \M$ and let $\gamma \colon [a,b] \to \M$ be a piecewise smooth curve joining $x$ to $y$. Then the length of the curve $\gamma$ is defined by 
$L(\gamma) \coloneq \int_a^b \|\gamma^{\prime }(t)\| dt$, where $\gamma^{\prime }(t) \in T_{\gamma(t)}\M$ is a tangent vector at the point $\gamma(t) \in \M$. The minimal length of all such curves joining $x$ to $y$ is called the Riemannian distance. It induces the original topology on $\M$ and is denoted by $d(x,y)$. 
In other words, $d(x, y) \coloneq \inf\{L(\gamma) \colon \gamma \in C_{xy}\}$, where $C_{xy}$ denotes the set of all continuously differentiable curves $\gamma \colon [0,1]\to \M$ such that $\gamma(0)=x$ and $\gamma(1)=y$.

A single-valued smooth vector field on a differentiable manifold $\M$ is a smooth mapping $A\colon\M \to T\M$ such that, for each $x \in \M$, 
there exists a tangent vector $A(x) \in T_{x}\M$.
Let $\triangle $ be the Levi-Civita connection associated with the Riemannian manifold $M$.
A vector field $A$ along $\gamma$ is said to be parallel if $\triangle_{\gamma^{\prime }}A=0$. 
If $\gamma^{\prime }$ is parallel along $\gamma$, that is, $\triangle_{\gamma^{\prime }} \gamma^{\prime }=0$, then $\gamma$ is said to be a geodesic. In this case, $\|\gamma^{\prime }\|$ is constant. Moreover, if $\|\gamma^{\prime }\|=1$, then $\gamma$ is called a unit-speed geodesic. A geodesic joining $x$ to $y$ in the Riemannian manifold $\M$ is said to be a minimal geodesic if its length is equal to $d(x,y)$. We denote the geodesic joining $x$ and $y$ by $\gamma(x,y; \cdot)$, that is,
$\gamma(x,y; \cdot) \colon [a,b] \to \M$ is such that, for $a,b \in \R$, we have $\gamma(x,y;a)=x$ and $\gamma(x,y;b)=y$.

The Riemannian metric induces a mapping \( f \mapsto \grad f \) that associates to each differentiable function \( f \colon \M \to \R \) 
a unique vector field \( \grad f \colon \M \to T\M \), called the gradient of \( f \), defined by
\(
\langle \grad f(x), X(x) \rangle = df_x(X(x))
\)
for all vector fields \( X \colon \M \to T\M \) and all \( x \in \M \). Here, \( df_x \) denotes the differential of \( f \) at the point \( x \in \M \).

A Riemannian manifold $\M$ is said to be complete if for any $x\in\M$, each geodesic emanating from $x$ is defined on the entire real line $\R$.
By the Hopf-Rinow theorem \cite[Theorem 1.1]{Sakai}, if $\M$ is complete, then each pair of points in $\M$ can be joined by a minimal geodesic. 
Moreover, $(\M,d)$ is a complete metric space. If $\M$ is a complete Riemannian manifold, then the exponential map $\exp_x \colon T_x\M \to \M$ at $x\in \M$ is defined
by $\exp_xv=\gamma_{x,v}(1) \text{ for all } v\in T_x\M$,
where $\gamma_{x,v}\colon \R \to \M$ is the unique geodesic starting from $x$ with velocity $v$, that is, $\gamma_{x,v}(0)=x$ and $ \gamma'_{x,v}(0)=v$. It is known that
$\exp_x(tv)=\gamma_{x,v}(t)$ for each real number $t$ and $\exp_x(\mathbf{0})=\gamma_{x,\mathbf{0}}(0)=x$, where $\mathbf{0}$ is the zero tangent vector.

\begin{definition}[\cite{Martin2}]
	A complete simply connected Riemannian manifold of nonpositive sectional curvature is called an Hadamard manifold.
\end{definition}

%
%
%

\begin{definition}[\cite{Martin2}]
	A nonempty subset $C$ of an Hadamard manifold $\M$ is said to be geodesically convex if, for any two points $x$ and $y$ in $C$, the geodesic joining $x$ to $y$ is contained in $C$, that is, if $\gamma(x,y; \cdot)\colon[a,b]\to \M$ is the geodesic such that $x=\gamma(x,y;a)$ and $y=\gamma(x,y;b)$, then
	$\gamma((1-t)a+tb) \in C$ for all $t\in [0,1]$.
\end{definition}

\begin{definition}[\cite{Martin2}]
	Let $\M$ be an Hadamard manifold. A function $f\colon\M \to \R$ is said to be geodesically convex
	if, for any geodesic $\gamma(x,y; \cdot)\colon[a,b] \to \M$,
	the composite function $f\circ \gamma(x,y;\cdot)\colon[a,b] \to \R$ is convex, that is,
	\begin{equation*}
		f\circ \gamma(x,y;(1-t)a+tb) \leq (1-t)(f\circ \gamma(x,y;a))+ t(f \circ
		\gamma(x,y;b))
	\end{equation*}
	for any $a,b \in \R$, $x,y \in \M$ and $t\in [0,1]$.
\end{definition}

\begin{lemma}[\cite{Martin2}]\label{exp-rem}
	Let $\{x_n\}$ be a sequence in an Hadamard manifold $\M$ such that $x_n \to x\in \M$. Then, for any $y \in \M$, we have
	$$\exp_{x_n}^{-1}y \to \exp_x^{-1}y \mbox{ and }  \exp_y^{-1}x_n \to \exp_y^{-1}x.$$
\end{lemma}

\begin{remark} [\cite{Martin2}] 
	Let $\M$ be an Hadamard manifold and let $x,y,z\in \M$. Then
	\begin{equation}\label{triangle-ineq}
		d^2(x,y) +d^2 (y,z)- d^2(z, x) \leq 2 \langle \exp_y^{-1}x, \exp_y^{-1}z \rangle.
	\end{equation}
\end{remark}

\begin{definition}[\cite{Udriste}]
	Let $\M$ be an Hadamard manifold and let $C$ be a nonempty,
	geodesically convex subset of $\M$. Let $f \colon C \to \R$ be a real-valued function. Then the subdifferential $\partial f \subset T\M$ of $f$ at $x$ is defined by
	\begin{equation*}\label{Subdiffeq2}
		\partial f(x) \coloneqq \left\{v\in T_x\M \colon  \langle v, \exp_x^{-1}y \rangle \leq f(y)-f(x) \text{ for all } y \in C\right\}.
	\end{equation*}
\end{definition}

%

\begin{definition}[\cite{Jost1995}]
	Let $C$ be a nonempty, closed and geodesically convex subset of an Hadamard manifold $\M$ and let $f\colon C \to (-\infty,\infty]$ be a proper, geodesically convex and lower semicontinuous function. The proximal operator of $f$ is defined by
	$$\prox_{\lambda f}(x) \coloneq
	\amin_{y\in C}\left(f(y)+\frac{1}{2\lambda}d^2(y,x)\right) \text{ for all } x\in C \text{ and for any } \lambda \in (0,\infty).$$
\end{definition}


\begin{lemma}[\cite{Ferreira2002}, Theorem 5.1] \label{prox-lem}
	Let $\M$ be an Hadamard manifold and let $f \colon \M \to \R$ be a geodesically convex function. Let $\{x_n\}$ be the sequence generated by the proximal point algorithm, 
	that is, 
	\[x_{n+1}=\amin_{y\in \M}\left\{f(y)+\frac{1}{2\lambda_n}d^2(x_n,y)\right\} \text{ for all } n \in \N_0=\N\cup \{0\}\]
	with the initial point $x_0\in \M$ and $\lambda_n\in (0,\infty)$. Then the sequence $\{x_n\}$ is well defined and characterized by $\frac{1}{\lambda_n}\exp_{x_{n+1}}^{-1}x_n \in \partial f(x_{n+1}).$
\end{lemma}

\begin{remark} \label{subdif-rem}
	In view of  \cite[Lemma 4.2]{Ferreira2002}, $\prox_{\lambda f}(x)$ is single-valued. By the definition of $\partial f(x_{n+1})$ and Lemma \ref{prox-lem}, we have
	\[\frac{1}{\lambda_n}\langle \exp_{x_{n+1}}^{-1}x_n, \exp_{x_{n+1}}^{-1}x \rangle \leq f(x)-f(x_{n+1}) \text{ for all } x\in \M.\]
\end{remark}

%
%
%

	\begin{remark}\label{dist-grad-rem}
		For fixed $x\in\M$,
		\(
		\grad d(y,x)=-\frac{\exp_y^{-1}x}{d(y,x)}\), provided \(y\neq x\) {}{and 
	\(
		\grad d^2(y,x)=-2\exp_y^{-1}x
		\) for all 	 \(y\in\M\).}
\end{remark}

\begin{definition}[\cite{Bento2022AOR, Tan2024}]
	Let $C$ be a nonempty, closed and geodesically convex subset of an Hadamard manifold $\M$.	A bifunction $F \colon C\times C \to \R$ is said to be
	\begin{enumerate}[label=\textup{(\roman*)}]
		\item monotone if 
		\(F(x,y)+F(y,x)\leq 0\) for all $x,y\in C$;
		\item pseudomonotone if  
		\(F(x,y)\geq 0\) implies \(F(y,x)\leq 0\) for all $x,y\in C$;
		\item $\alpha$-strongly monotone if there exists a constant $\alpha>0$ such that
		\(F(x,y)+F(y,x)\leq -\alpha d^2(x,y)\) for all $x,y\in C$;
		\item $\beta$-strongly pseudomonotone if there exists a constant $\beta>0$ such that 
		\(F(x,y)\geq 0\) implies \(F(y,x)\leq -\beta d^2(x,y)\) for all $x,y\in C$.
	\end{enumerate}
\end{definition}


The Busemann functions, originally introduced by Herbert Busemann to formulate the parallel axiom in certain classes of metric spaces~\cite{Busemann1955}, have since found applications far beyond their original purpose; see, for example,~\cite{Busemann1993, Li1987, Shiohama, Sormani} and the references therein.

\begin{definition}[\cite{Busemann1993}]
	Let $\M$ be an Hadamard manifold and let $\gamma_{z,x} \colon [0,\infty) \to \M$ be a geodesic ray parameterized by arc length, starting from $z$ and passing through $x$.  
	The Busemann function associated with $\gamma_{z,x}$ is defined by
	\[
	b_{\gamma_{z,x}}(y) \coloneqq \lim_{t \to +\infty} \left[ d(y, \gamma_{z,x}(t)) - t \right]
	 \text{ for all } y \in \M.
	\]
\end{definition}

Recently, Busemann functions have been studied in the context of optimization; see \cite{Bento2022, Bento2023}. In this direction, Bento et al.~\cite{Bento2022} introduced the following resolvent associated with a bifunction, which is defined using a Busemann function.

\begin{definition}[\cite{Bento2022}]
	Let $C$ be a nonempty, closed, and geodesically convex subset of an Hadamard manifold $\M$, and let $F \colon C \times C \to \R$ be a bifunction. The Busemann resolvent $J_\lambda^F \colon C \to \M$ associated with $F$ is defined by
	\begin{equation}\label{ResolF}
		J_\lambda^F(x) \coloneqq \{z \in C \colon  \lambda F(z,y) + d(z,x)b_{\gamma_{z,x}}(y) \geq 0 \text{ for all } y \in C\},
	\end{equation}
	for all $x \in C$ and any $\lambda \in (0,\infty)$.
\end{definition}

By \cite[Lemma 3.4]{Ballmann}, the Busemann function $b_{\gamma_{x,z}}(\cdot)$ is $1$-Lipschitz continuous, geodesically convex, and satisfies $\|\operatorname{grad} b_{\gamma_{x,z}}(\cdot)\| = 1$. Moreover, as noted in~\cite[page 273]{Bridson}, in the Euclidean setting one has
\begin{equation}\label{Hirstoagaeqn}
{d(z,x)\, b_{\gamma_{z,x}}(y) = \langle z - x, y - z \rangle.}
\end{equation}
This identity shows that the additional term appearing in \eqref{ResolF} is a natural extension of the regularization term introduced in~\cite{Hirstoaga}.

On the other hand, it is known that any linear  function on the Poincar\'e plane must be constant; see~\cite[Theorem 2.2, p.~299]{Udriste}. This fact reflects the absence of nontrivial affine functions in negatively curved spaces. In line with~\cite[Theorem 1]{Bento2023}, this result extends to general Hadamard manifolds with negative sectional curvature. Consequently, Busemann functions provide a meaningful substitute for affine functions in linear settings. As a result, the resolvent defined in~\eqref{ResolF}, which incorporates a Busemann function, serves as a natural and well-adapted analogue of the classical linear resolvent in the Riemannian framework.

\begin{proposition}\label{PropBuse}
	Let $\M$ be an Hadamard manifold and let $x,z\in \M$. Then 
	\begin{equation}\label{BuseMotive}
		d(x,z)b_{\gamma_{z,x}}(y)\geq -\langle \exp_z^{-1}x, \exp_z^{-1}y \rangle \text{ for all } y \in \M.
			\end{equation}
\end{proposition}
\begin{proof}
	Since  $\gamma_{z,x}$ is the geodesic ray parameterized by arc length, we have
	\[\gamma_{z,x}(t)=\exp_z \left(\frac{t}{d(x,z)}\exp_z^{-1}x\right) \text{ for all } t \in [0,\infty).\]
	Using \eqref{triangle-ineq}, we get
	\begin{align}
		d^2(y,\gamma_{z,x}(t))&\geq d^2(z,y)+d^2(\gamma_{z,x}(t),z)-2\langle \exp_z^{-1}\gamma_{z,x}(t), \exp_z^{-1}y\rangle\nonumber\\
		&= d^2(z,y)+\frac{t^2}{d^2(x,z)} d^2(z,x)-2\frac{t}{d(x,z)} \langle \exp_z^{-1}x, \exp_z^{-1}y\rangle\nonumber\\
		&= d^2(z,y)-\frac{2t}{d(x,z)} \langle \exp_z^{-1}x, \exp_z^{-1}y\rangle+t^2. \label{Bregeqq1}
	\end{align}
	Using the same argument as in  \cite[Example 8.24 (3), Page  274]{Bridson} and \eqref{Bregeqq1}, we have
	\begin{align*}
		b_{\gamma_{z,x}}(y)&=\lim_{t \to +\infty} \left[ d(y, \gamma_{z,x}(t)) - t \right]=\lim_{t \to +\infty}\frac{1}{2t}\left( d^2(y,\gamma_{z,x}(t))-t^2\right)\\
		&\geq \lim_{t \to +\infty}\frac{1}{2t} \left(d^2(z,y)-\frac{2t}{d(x,z)} \langle \exp_z^{-1}x, \exp_z^{-1}y\rangle\right)\\
		&= -\frac{1}{d(x,z)} \langle \exp_z^{-1}x, \exp_z^{-1}y\rangle.
	\end{align*}
	Therefore 
	\(d(x,z)b_{\gamma_{z,x}}(y)\geq- \langle \exp_z^{-1}x, \exp_z^{-1}y \rangle\).
\end{proof}

\begin{assumption} Let $C$ be a nonempty, closed and geodesically convex subset of an Hadamard manifold $\M$ and let $F \colon C \times C \to \R$ be a bifunction.
	We consider the following assumptions:
	\begin{enumerate}
		\item[$(A1)$] $F$ is   
		monotone and $F(x,x)=0$ for all $x\in C$;
		\item[$(A2)$] $F(x,\cdot)$ is geodesically convex and subdifferentiable on $C$ for each $x\in C$;
		\item[$(A3)$] $F(\cdot,y)$ is upper semicontinuous on $C$ for each $y\in C$, that is, $\limsup_{n\to \infty} F(x_n,y)\leq F(x,y)$ for each $y\in C$ and each $\{x_n\} \subset C$ with $x_n\to x$;
		\item[$(A4)$] Given a fixed $z_0\in \M$, consider a sequence $\{z_n\}\subset C$ such that $d(z_n,z_0) \to \infty$ as $n \to \infty$.  
		Then there exist $x^*\in C$ and $n_0\in \N$ such that
		$$F(z_n,x^*)\leq 0 \text{ for all } n \geq n_0.$$
	\end{enumerate}
\end{assumption}



Let $\lambda\in (0,\infty)$ and $x\in C$ and let $F \colon C\times C \to \R$ be a bifunction. 
For each $x\in C$, consider the set
{}{\[K_x \coloneqq \left\{z\in C \colon  \lambda F(z,y)+\frac{1}{2}(d^2(y,x)-d^2(z,x))\geq 0  \text{ for all } y \in C\right\}.\]}

{}{Since \cite[Corollary 1]{Bento2022AOR} holds for every positive
parameter, it also applies with the factor $\frac12$ by rescaling.}
Thus, combining this result with \cite[Theorem 4.1]{Bento2022}
and \cite[Theorem 1.1]{Bento2024}, we arrive at the following result:

\begin{theorem} \label{ExistThm}
	Let $C$ be a nonempty, closed and geodesically convex subset of an Hadamard manifold $\M$ and let $F \colon C \times C \to \R$ be a bifunction such that Assumptions $(A1)$ -- $(A4)$ hold. Let $\lambda \in (0,\infty)$. Then
		\begin{enumerate}[label=\textup{(\roman*)}]
		\item $\ep(F,C)\neq \emptyset$,
		\item $J_\lambda^F(x)\neq \emptyset$ for each $x\in C$,
		\item $J_\lambda^F$ is single-valued,
		\item the set $K_x\neq \emptyset$ for each $x\in C$,
		\item  the set $K_x$ is a singleton for each $x\in C$.
			\end{enumerate}
\end{theorem}

\begin{remark}
	Let $\M$ be an Hadamard manifold with zero sectional curvature. Then
	\begin{enumerate}[label=\textup{(\roman*)}]
		\item Let $x, z \in \M$. Using Proposition~\ref{PropBuse}, we have
		\[
		d(x,z)\, b_{\gamma_{z,x}}(y)
		= - \langle \exp_z^{-1} x,\; \exp_z^{-1} y \rangle
		\text{ for all } y \in \M.
		\]
		\item Suppose that $F$ satisfies Assumptions $(A1)$–$(A4)$. Then, by \cite[Theorem~4.7(ii)]{Colao2012}, the resolvent $J_\lambda^F$ is firmly nonexpansive.
	\end{enumerate}
\end{remark}


	
	\begin{theorem} [\cite{Tan2024}, Theorem 4.1] \label{SPMthm}
		Let $C$ be a nonempty, closed and geodesically convex subset of an Hadamard manifold $\M$ and let $F \colon C\times C\to \R$ be a $\beta$-strongly {pseudomonotone} bifunction such that $F(x,x)=0$ and  Assumptions $(A2)$--$(A3)$ hold. Then equilibrium problem \eqref{equib-prob} has a unique solution.
	\end{theorem}
	
	\begin{definition}[\cite{Ferreira2002}]
		Let $C$ be a nonempty subset of a complete metric space $X$. A sequence $\{x_n\}\subset X$ is called Fej\'er monotone with respect to $C$ if
		\[d(x_{n+1},y)\leq d(x_n,y) \text{ for all } y \in C \text{ and for all } n \in \N_0.\]
	\end{definition}
	
	\begin{lemma} [\cite{Ferreira2002}]\label{FejerLemma}
		Let $C$ be a nonempty subset of a complete metric space $X$. If $\{x_n\}\subset X$ is Fej\'er monotone with respect to $C$, then $\{x_n\}$ is bounded. In addition, if  $\{x_n\}$ has a cluster point  $x$  in  $C$, then $\{x_n\}$ converges to $x$.
	\end{lemma}

	\section{Regularized extragradient methods for solving equilibrium problems}

		%

	In this section we present two regularized extragradient methods for solving  equilibrium problem \eqref{equib-prob} on Hadamard manifolds.

	\begin{Algorithm} \label{MAlg} Let $C$ be a nonempty, closed and  geodesically convex subset of an Hadamard manifold $\M$ and let $F \colon C \times C \to \R$ be a bifunction satisfying Assumptions  $(A1)$--$(A4)$ .\\
		\textbf{Initialization:} Choose $x_0\in C$ and a sequence $\{\lambda_n\}\subset(0,\infty)$
		such that ${0<\lambda \leq \lambda_n}$ for all $n\in\N_0$, where $\lambda>0$ is fixed.\\
		\textbf{Iterative Steps:} Given the current iterate $x_n$, compute the next iterate as follows:
		\begin{equation*}
			\begin{cases}
				y_n \in \left\{x\in C \colon \lambda_nF(x,y)+  d(x,x_n)b_{\gamma_{x,x_n}}(y)\geq 0 \text{ for all } y\in C\right\}=J_{\lambda_n}^F(x_n),\\
				x_{n+1} = \amin_{y \in C} \left\{ F(y_n, y) + \frac{1}{2\lambda_n} d^2(x_n, y) \right\}= \prox_{ \lambda_n F(y_n,\cdot)}(x_n) \text{ for all } n \in \N_0.
			\end{cases}
		\end{equation*}
	\end{Algorithm}

	\begin{Algorithm} \label{MAlg2} Let $C$ be a nonempty, closed and geodesically convex subset of an Hadamard manifold $\M$ and let $F \colon C \times C \to \R$ be a bifunction satisfying Assumptions  $(A1)$--$(A4)$.\\
		\textbf{Initialization:} Choose $x_0\in C$ and a sequence $\{\lambda_n\}\subset(0,\infty)$
		such that ${0<\lambda \leq \lambda_n}$ for all $n\in\N_0$, where $\lambda>0$ is fixed.\\
		\textbf{Iterative Steps:} Given the current iterate $x_n$, compute the next iterate as follows:
		\begin{equation*}
			\begin{cases}
			{}{	y_n \in  \left\{x\in C \colon \lambda_n F(x,y)+\frac{1}{2}(d^2(y,x_n)-d^2(x,x_n))\geq 0 \text{ for all } y\in C\right\},}\\
				x_{n+1} = \amin_{y \in C} \left\{ F(y_n, y) + \frac{1}{2\lambda_n} d^2(x_n, y) \right\} =\prox_{ \lambda_n F(y_n,\cdot)}(x_n) \text{ for all } n \in \N_0.
			\end{cases}
		\end{equation*}
	\end{Algorithm}
	
	\begin{remark}
	{}{	In Algorithms~\ref{MAlg} and~\ref{MAlg2}, $y_n$ is uniquely determined by Theorem~\ref{ExistThm},
		while $x_{n+1}$ exists and is unique by the standard properties
		of the proximal operator. Hence, both algorithms are well defined.}
	\end{remark}
	
\begin{remark}
	Note that in a linear space we have
	\(-\langle \exp_z^{-1}x,\exp_z^{-1}y\rangle =\langle z-x,y-z\rangle.\)
	Hence, the term $-\langle \exp_z^{-1}x,\exp_z^{-1}y\rangle$ may be regarded as a natural extension of the classical regularization term $\langle z-x,y-z\rangle$ from linear spaces to Hadamard manifolds. However, in general, the mapping
	\(y\mapsto \langle \exp_z^{-1}x,\exp_z^{-1}y\rangle\)
	is not geodesically convex on an Hadamard manifold; see \cite{kristaly2016convexities}. Therefore, this regularization does not, in general, guarantee the existence of solutions for the associated regularized equilibrium problem.
	
	On the other hand, using \eqref{BuseMotive}, we have
	\(d(x,z)b_{\gamma_{z,x}}(y)
	\geq
	-\langle \exp_z^{-1}x,\exp_z^{-1}y\rangle.\)
	Moreover, the function
	\(
	y\mapsto d(x,z)b_{\gamma_{z,x}}(y)
	\)
	is geodesically convex, which makes it a suitable regularization term for our framework.
	
	{}{Similarly, the term
	\(
	\frac12(d^2(y,z)-d^2(x,y)-d^2(x,z))
	\)
	is generally not geodesically convex with respect to $y$ because of the presence of the term $-d^2(x,y)$. However,
	\(
	\frac12(d^2(y,z)-d^2(x,z))
	\geq
	\frac12(d^2(y,z)-d^2(x,z)-d^2(x,y))
	\geq
	-\langle \exp_z^{-1}x,\exp_z^{-1}y\rangle,
	\)
	and the mapping
	\(
	y\mapsto d^2(y,z)-d^2(x,z)
	\)
	is geodesically convex with respect to $y$. Therefore, it provides another suitable candidate for regularization in the equilibrium problem framework.}
		\end{remark}
	
	\begin{remark}
	 {}{The definition of the
	Busemann function and the triangle inequality yield
	\(
	b_{\gamma_{z,x}}(y)
	\leq d(y,x)-d(z,x).
	\)
	Consequently,
	\[
	\begin{aligned}
		d(z,x)b_{\gamma_{z,x}}(y)
		\leq d(z,x)\bigl(d(y,x)-d(z,x)\bigr)\leq \frac12\bigl(d^2(y,x)-d^2(z,x)\bigr),
	\end{aligned}
	\]
	where the last inequality follows from
	$\bigl(d(y,x)-d(z,x)\bigr)^2\geq0$.
	Combining these estimates, we obtain
	\[
	-\langle \exp_z^{-1}x,\exp_z^{-1}y\rangle
	\leq d(z,x)b_{\gamma_{z,x}}(y)
	\leq \frac12\bigl(d^2(y,x)-d^2(z,x)\bigr)
 \text{ for all }y\in\M.
	\]
	Thus, both regularization terms are geodesically convex
	and admit a direct comparison, although they arise from different geometric constructions on Hadamard manifolds.
	}
\end{remark}

	
	Next, we prove an important lemma that is essential for the convergence analysis of Algorithm \ref{MAlg}.

	\begin{lemma}\label{Key-lemma} Let $C$ be a nonempty, closed and  geodesically convex subset of an Hadamard manifold $\M$ and let $F \colon C \times C \to \R$ be a bifunction satisfying
		Assumptions $(A1)$--$(A4)$. Let $\{x_n\}$ and $\{y_n\}$ be sequences generated by Algorithm \ref{MAlg}.  Then, for each $x^* \in \ep(F,C)$, we have
		\begin{equation}\label{key-lem-eq}
			d^2(x_{n+1}, x^*) \leq d^2(x_n, x^*) -  d^2(x_n, y_n) - d^2(x_{n+1}, y_n).
		\end{equation}
		Furthermore, $\{x_n\}$ is Fej\'er monotone with respect to  $\ep(F,C)$ and the sequences $\{x_n\}$ and $\{y_n\}$ are bounded.
	\end{lemma}

	\begin{proof}
		Using the definition of $x_{n+1}$ in Algorithm \ref{MAlg} and Remark \ref{subdif-rem}, we have
		\begin{equation}\label{lem-prox-eq}
			\lambda_n (F(y_n, y) - F(y_n, x_{n+1})) \geq  \langle \exp^{-1}_{x_{n+1}} x_n, \exp^{-1}_{x_{n+1}} y \rangle \text{ for all } y \in C.
		\end{equation}
		Using  the definition of $y_{n}$ and $J_{\lambda_n}^F(x_n)$, we have
		\begin{equation}\label{lem-resolv-eq1}
			0\leq \lambda_n F(y_n,y)+ d(y_n,x_n) b_{\gamma_{y_n,x_n}}(y) \text{ for all } y \in C.
		\end{equation}
		Since $b_{\gamma_{y_n,x_n}}(y) \leq d(y,x_n)-d(x_n,y_n)$ for all $y\in C$, by \eqref{lem-resolv-eq1}, we have
		\begin{equation}\label{lem-resolv-eq2}
			0\leq \lambda_n F(y_n,y)+ d(y_n,x_n) [d(y,x_n)-d(x_n,y_n)] \text{ for all } y \in C.
		\end{equation}
		If $y_n=x_n$, the defining inequality for $y_n$ gives
		$F(y_n,y)\geq0$ for every $y\in C$. Thus,
		$x_n\in\operatorname{EP}(F,C)$, and
		\begin{equation}\label{lem-resolv-eq33"}
			\langle \exp_{y_n}^{-1}x_n,\exp_{y_n}^{-1}x\rangle
		\leq \lambda_n F(y_n,x)
		\text{ for all } x\in C.
		\end{equation}
		Suppose now that $y_n\neq x_n$. Define $G_n(y_n,\cdot) \colon C \to \R$  by
		\[G_n(y_n,y)\coloneqq \lambda_n F(y_n,y)+ d(y_n,x_n) [d(y,x_n)-d(x_n,y_n)] \text{ for all } y \in C.\]
		Then $G_n(y_n,y_n)=0$. It follows from \eqref{lem-resolv-eq2}  that 
		$$y_n=\amin_{y\in C}G_n(y_n,y),$$
		which implies that  $0\in \partial G_n({y_n,y_n})$. It follows from Remark \ref{dist-grad-rem} that
		\begin{eqnarray*}
			0\in \partial G_n({y_n,y_n})&=&\partial (\lambda_n F(y_n,y)+ d(y_n,x_n) [d(y,x_n)-d(x_n,y_n)] )\big|_{y_n}\\
			&=& \lambda_n \partial F(y_n,y_n)+\partial (d(y_n,x_n) [d(y,x_n)-d(x_n,y_n)] )\big|_{y_n}\\
			&=& \lambda_n\partial F(y_n,y_n)+  d(y_n,x_n) (-d^{-1}(y_n,x_n)\exp_{y_n}^{-1}x_n)\\
			&=& \lambda_n\partial F(y_n,y_n)-\exp_{y_n}^{-1}x_n.
		\end{eqnarray*}
		Thus $(1/\lambda_n)\exp_{y_n}^{-1}x_n \in \partial F(y_n,y_n)$. By the definition of $ \partial F(y_n,y_n)$, we have
		\begin{align}\label{lem-resolv-eq33}
			\langle \exp_{y_n}^{-1}x_n, \exp_{y_n}^{-1}x \rangle \leq \lambda_n(F(y_n,x)-F(y_n,y_n))= \lambda_n F(y_n,x) \text{ for all } x\in C.
		\end{align}
		Setting $x = x_{n+1} \in C$ in \eqref{lem-resolv-eq33"} and \eqref{lem-resolv-eq33}, we obtain
		\begin{equation}\label{lem-resolv-eq3}
			\lambda_n  F(y_n, x_{n+1}) \geq  \langle \exp^{-1}_{y_n} x_n, \exp^{-1}_{y_n} x_{n+1} \rangle.
		\end{equation}
		Using \eqref{lem-prox-eq} and \eqref{lem-resolv-eq3}, we find that
		\begin{align}
			\langle \exp_{y_{n}}^{-1}x_n,\exp_{y_{n}}^{-1}x_{n+1} \rangle+ \langle \exp_{x_{n+1}}^{-1}x_n,\exp_{x_{n+1}}^{-1}y \rangle & \leq \lambda_n (F(y_n,x_{n+1})+F(y_n,y)-F(y_n,x_{n+1}))\nonumber\\
			&= \lambda_n F(y_n,y)\text{ for all } y \in C. \label{lem-prox-eq4}
		\end{align}
		Using \eqref{triangle-ineq} for $x_n,x_{n+1},y \in \M$, one sees that
		\begin{align}\label{tngl-eq1}
			2\langle \exp_{x_{n+1}}^{-1}x_n,\exp_{x_{n+1}}^{-1}y \rangle \geq d^2(x_n,x_{n+1})+d^2(x_{n+1},y)-d^2(x_n,y) \text{ for all } y \in C.
		\end{align}
		Again, using \eqref{triangle-ineq} for $x_n,y_n,x_{n+1}\in \M$, one sees that
		\begin{align}\label{tngl-eq2}
			2\langle \exp_{y_{n}}^{-1}x_n,\exp_{y_{n}}^{-1}x_{n+1} \rangle \geq d^2(x_n,y_{n})+d^2(x_{n+1},y_n)-d^2(x_n,x_{n+1}).
		\end{align}
		Using \eqref{lem-prox-eq4}, \eqref{tngl-eq1} and \eqref{tngl-eq2}, we have
		\[	d^2(x_n,y_n)+ d^2(x_{n+1},y_n)- d^2(x_n,x_{n+1})
		+ d^2(x_n,x_{n+1})+d^2(x_{n+1},y)-d^2(x_n,y) \leq 2\lambda_n F(y_n,y),\]
		which implies that
		\begin{eqnarray}\label{lem-prof-eq1}
			d^2(x_{n+1},y)&\leq& d^2(x_n,y)- d^2(x_n,y_n)- d^2(x_{n+1},y_n) +2 \lambda_n F(y_n,y) \text{ for all } y \in C.
		\end{eqnarray}
		Let $x^*\in \ep(F,C)$ and put $y=x^*$ in \eqref{lem-prof-eq1}. Since $y_n \in C$, we obtain $F(x^*,y_n)\geq 0$. When combined with the monotonicity of $F$,  this  yields $F(y_n,x^*)\leq 0$. Thus, using \eqref{lem-prof-eq1}, we have
		\begin{equation}\label{lem-prof-eq1s}
			d^2(x_{n+1},x^*)\leq d^2(x_n,x^*) -d^2(x_n,y_n)-d^2(x_{n+1},y_n),
		\end{equation}
		which implies that
		\begin{equation}\label{key-eq}
			d(x_{n+1},x^*)\leq d(x_n,x^*) \text{ for all }  n\in \N_0.
		\end{equation}
		Since $x^* \in \ep(F,C)$ is arbitrary, this means that $\{x_n\}$ is Fej\'er monotone with respect to $\ep(F,C)$.
		By Lemma \ref{FejerLemma}, the sequence $\{x_n\}$ is bounded, and $\lim_{n \to \infty} d(x_n,x^*)$ exists. By taking the limit as $n \to \infty$ in \eqref{lem-prof-eq1s}, one arrives at
		\begin{equation}\label{lim-eq}
			\lim_{n \to \infty} d(x_n,y_n)=0.
		\end{equation}
		Thus, the sequence $\{y_n\}$ is also bounded.
	\end{proof}

	\begin{theorem}
		Let $C$ be a nonempty, closed and geodesically convex subset of an Hadamard manifold $\M$ and let $F \colon C \times C \to \R$ be a bifunction satisfying
		Assumptions $(A1)$--$(A4)$.	Let $\{x_n\}$ be the sequence generated by Algorithm \ref{MAlg}, where ${0<\lambda \leq \lambda_n}$ for some $\lambda \in (0,\infty)$. Then $\{x_n\}$ converges to an element of $\ep(F,C)$.
	\end{theorem}
	
	\begin{proof}
		In view of Lemma  \ref{Key-lemma}, the sequence  \( \{x_n\} \) is Fej\'{e}r monotone with respect to the solution set \( \ep(F,C)\) and the sequences $\{x_n\}$ and $\{y_n\}$ are bounded. 	Since \( \{x_n\} \) is bounded, the Hopf--Rinow theorem implies that $\{x_n\}$ has a
		convergent subsequence, say $x_{n_k}\to\bar x$ as \( k \to \infty \).
		From \eqref{lim-eq}, we have $\lim_{k \to \infty}d(x_{n_k},y_{n_k})=0$. Thus, \( y_{n_k} \to \bar{x} \) as \( k \to \infty \).
		Using \eqref{lem-resolv-eq33"} and \eqref{lem-resolv-eq33}, we see that 
		\[
		F(y_{n_k}, y) \geq \frac{1}{\lambda_{n_k}}  \langle \exp_{y_{n_k}}^{-1}(x_{n_k}), \exp_{y_{n_k}}^{-1}(y) \rangle \text{ for all } y \in C.
		\]
	{	Using Assumption $(A3)$, we have
			\begin{equation}\label{Npfeq11}
				F(\bar{x}, y) \geq \limsup_{k\to \infty} F(y_{n_k},y) \geq  \limsup_{k\to \infty}   \frac{1}{\lambda_{n_k}}  \langle \exp_{y_{n_k}}^{-1}(x_{n_k}), \exp_{y_{n_k}}^{-1}(y) \rangle \text{ for all } y \in C.
			\end{equation}
			Using the condition $0<  \lambda \leq \lambda_n $, we get
			\[\left|\frac{1}{\lambda_{n_k}}  \langle \exp_{y_{n_k}}^{-1}(x_{n_k}), \exp_{y_{n_k}}^{-1}(y) \rangle\right| \leq \frac{1}{\lambda}  |\langle \exp_{y_{n_k}}^{-1}(x_{n_k}), \exp_{y_{n_k}}^{-1}(y) \rangle|.\]
			By Lemma \ref{exp-rem}, we have
			\[\lim_{k\to \infty}\langle \exp_{y_{n_k}}^{-1}(x_{n_k}), \exp_{y_{n_k}}^{-1}(y) \rangle=0.\]
			Consequently 
			\begin{equation}\label{Npfeq21}
				\lim_{k\to \infty}\frac{1}{\lambda_{n_k}}  \langle \exp_{y_{n_k}}^{-1}(x_{n_k}), \exp_{y_{n_k}}^{-1}(y) \rangle=0.
			\end{equation}
			Using \eqref{Npfeq11} and \eqref{Npfeq21}, we obtain
			\[
			F(\bar{x},y)
			\geq
			\limsup_{k\to\infty} F(y_{n_k},y)
			\geq 0
			\text{ for all }y\in C,
			\]
			which implies that \(\bar{x} \in \ep(F,C) \).} 	Thus, $\{x_n\}$ is Fej\'er monotone with respect to
			$ \ep(F,C)$ and has a cluster point in this set.
			By Lemma~\ref{FejerLemma}, the entire sequence $\{x_n\}$ converges to $\bar x$.
%
	\end{proof}

	We now turn to the convergence analysis of Algorithm~\ref{MAlg2}, for which the following lemma is essential.

	\begin{lemma}\label{Key-lemma2} Let $C$ be a nonempty, closed and  geodesically convex subset of an Hadamard manifold $\M$ and let $F \colon C \times C \to \R$ be a bifunction satisfying
		Assumptions $(A1)$--$(A4)$. Let $\{x_n\}$ and $\{y_n\}$ be the sequences generated by Algorithm \ref{MAlg2}.  Then, for each $x^* \in \ep(F,C)$, we have
		\begin{equation}\label{key-lem-eq2}
			d^2(x_{n+1}, x^*) \leq d^2(x_n, x^*) -  d^2(x_n, y_n) - d^2(x_{n+1}, y_n).
		\end{equation}
		Furthermore, $\{x_n\}$ is Fej\'er monotone with respect to the solution set $\ep(F,C)$ and the sequences $\{x_n\}$ and $\{y_n\}$ are bounded.
	\end{lemma}

	\begin{proof}
		According to the definition of $x_{n+1}$ in Algorithm \ref{MAlg2} and Remark \ref{subdif-rem}, we have
		\begin{equation}\label{lem-prox-eq2}
			\lambda_n (F(y_n, y) - F(y_n, x_{n+1})) \geq \langle \exp^{-1}_{x_{n+1}} x_n, \exp^{-1}_{x_{n+1}} y \rangle \text{ for all } y \in C.
		\end{equation}
		From the definition of $y_{n}$, it follows that
		\begin{equation}\label{lem-resolv-eq12}
			0\leq \lambda_nF(y_n,y)+\frac{1}{2}(d^2(y,x_n)-d^2(y_n,x_n)) \text{ for all } y \in C.
		\end{equation}
		Let $G_n(y_n,\cdot) \colon C \to \R$ be the function defined by
		\[G_n(y_n,y)=\lambda_n F(y_n,y)+ \frac{1}{2}(d^2(y,x_n)-d^2(y_n,x_n)) \text{ for all } y \in C.\]
		Then $G_n(y_n,y_n)=0$. Using \eqref{lem-resolv-eq12}, we find that 
		$$y_n=\amin_{y\in C}G_n(y_n,y),$$
		which implies that  $0\in \partial G_n({y_n,y_n})$. {}{It follows from Remark \ref{dist-grad-rem} that}
	{}{	\begin{align*}
			0\in \partial G_n({y_n,y_n})&=\partial (\lambda_n F(y_n,y)+ \frac{1}{2}(d^2(y,x_n)-d^2(y_n,x_n)) )\big|_{y_n}\\
			&=\lambda_n \partial F(y_n,y_n)+\partial \left(\frac{1}{2}(d^2(y,x_n)-d^2(y_n,x_n))\right)\Big|_{y_n}\\
			&= \lambda_n\partial  F(y_n,y_n)- \exp_{y_n}^{-1}x_n.
		\end{align*}}
		Thus $(1/\lambda_n)\exp_{y_n}^{-1}x_n \in \partial F(y_n,y_n)$. By the definition of $ \partial F(y_n,y_n)$, we have
		\begin{align}\label{lem-resolv-eq332}
			\frac{1}{\lambda_n} \langle \exp_{y_n}^{-1}x_n, \exp_{y_n}^{-1}x \rangle \leq F(y_n,x)-F(y_n,y_n)= F(y_n,x) \text{ for all } x\in C.
		\end{align}
		Setting $x = x_{n+1} \in C$ in \eqref{lem-resolv-eq332}, we obtain
		\begin{equation}\label{lem-resolv-eq32}
			\lambda_n F(y_n, x_{n+1}) \geq  \langle \exp^{-1}_{y_n} x_n, \exp^{-1}_{y_n} x_{n+1} \rangle.
		\end{equation}
		Using \eqref{lem-prox-eq2} and \eqref{lem-resolv-eq32}, we have
		\begin{align}
			\langle \exp_{y_{n}}^{-1}x_n,\exp_{y_{n}}^{-1}x_{n+1} \rangle+\langle \exp_{x_{n+1}}^{-1}x_n,\exp_{x_{n+1}}^{-1}y \rangle & \leq \lambda_n (F(y_n,x_{n+1})+F(y_n,y)-F(y_n,x_{n+1}))\nonumber\\
			&= \lambda_n F(y_n,y)\text{ for all } y \in C. \label{lem-prox-eq42}
		\end{align}
		Using \eqref{triangle-ineq} with $x_n,x_{n+1},y\in \M$, we obtain
		\begin{align}\label{tngl-eq12}
			2\langle \exp_{x_{n+1}}^{-1}x_n,\exp_{x_{n+1}}^{-1}y \rangle \geq d^2(x_n,x_{n+1})+d^2(x_{n+1},y)-d^2(x_n,y) \text{ for all } y \in C.
		\end{align}
		Again, applying \eqref{triangle-ineq} to $x_n,y_n,x_{n+1}\in \M$, we obtain
		\begin{align}\label{tngl-eq22}
			2\langle \exp_{y_n}^{-1}x_n,\exp_{y_n}^{-1}x_{n+1} \rangle \geq d^2(x_n,y_n)+d^2(x_{n+1},y_n)-d^2(x_n,x_{n+1}).
		\end{align}
		Using \eqref{lem-prox-eq42}, \eqref{tngl-eq12} and \eqref{tngl-eq22}, we have
		\[	d^2(x_n,y_n)+ d^2(x_{n+1},y_n)- d^2(x_n,x_{n+1})
		+ d^2(x_n,x_{n+1})+d^2(x_{n+1},y)-d^2(x_n,y) \leq 2\lambda_n F(y_n,y),\]
		which implies that
		\begin{eqnarray}\label{lem-prof-eq12}
			d^2(x_{n+1},y)&\leq& d^2(x_n,y)- d^2(x_n,y_n)- d^2(x_{n+1},y_n) +2 \lambda_n F(y_n,y) \text{ for all } y \in C.
		\end{eqnarray}
		Let $x^*\in \ep(F,C)$ and put $y=x^*$ in \eqref{lem-prof-eq12}. Since $y_n \in C$, we obtain $F(x^*,y_n)\geq 0$. When combined with the monotonicity of $F$, this shows that $F(y_n,x^*)\leq 0$. It now follows from  \eqref{lem-prof-eq12} that
		\begin{align}\label{key-eq2}
			d^2(x_{n+1},x^*)\leq d^2(x_n,x^*) -d^2(x_n,y_n)-d^2(x_{n+1},y_n),
		\end{align}
		which implies that
		\begin{equation}
			d(x_{n+1},x^*)\leq d(x_n,x^*) \text{ for all }  n\in \N_0.
		\end{equation}
		Since $x^* \in \ep(F,C)$ was arbitrary, this means that $\{x_n\}$ is Fej\'er monotone with respect to $\ep(F,C)$.
		By Lemma \ref{FejerLemma}, the sequence $\{x_n\}$ is bounded, and $\lim_{n \to \infty} d(x_n,x^*)$ exists. By taking the limit as $n \to \infty$ in \eqref{key-eq2}, one arrives at
		\begin{equation}\label{lim-eq2}
			\lim_{n \to \infty} d(x_n,y_n)=0.
		\end{equation}
		Thus, $\{y_n\}$ is also bounded.
	\end{proof}

	\begin{theorem}
		Let $C$ be a nonempty, closed and geodesically convex subset of an Hadamard manifold $\M$ and let $F \colon C \times C \to \R$ be a bifunction satisfying the
		Assumptions $(A1)$--$(A4)$.	Let $\{x_n\}$ be the sequence generated by Algorithm \ref{MAlg2}, where ${0<\lambda \leq \lambda_n}$ for some $\lambda \in (0,\infty)$. Then $\{x_n\}$ converges to an element of $\ep(F,C)$.
	\end{theorem}
	
	\begin{proof}
		It follows from  Lemma \ref{Key-lemma2} that \( \{x_n\} \) is Fej\'{e}r monotone with respect to the solution set \( \ep(F,C)\) and  the sequences $\{x_n\}$ and $\{y_n\}$ are bounded.
		Since \( \{x_n\} \) is bounded, the Hopf--Rinow theorem implies that $\{x_n\}$ has a
		convergent subsequence, say $x_{n_k}\to\bar x$ as \( k \to \infty \).  Using \eqref{lim-eq2}, we have $\lim_{k \to \infty}d(x_{n_k},y_{n_k})=0$. Thus, \( y_{n_k} \to \bar{x} \) as \( k \to \infty \).
		Using \eqref{lem-resolv-eq332}, we have
		\[
		F(y_{n_k}, y) \geq \frac{1}{\lambda_{n_k}}  \langle \exp_{y_{n_k}}^{-1}(x_{n_k}), \exp_{y_{n_k}}^{-1}(y) \rangle \text{ for all } y \in C.
		\]
{	Using Assumption $(A3)$, we have
		\begin{equation}\label{Npfeq1}
			F(\bar{x}, y) \geq \limsup_{k\to \infty} F(y_{n_k},y) \geq  \limsup_{k\to \infty}   \frac{1}{\lambda_{n_k}}  \langle \exp_{y_{n_k}}^{-1}(x_{n_k}), \exp_{y_{n_k}}^{-1}(y) \rangle \text{ for all } y \in C.
		\end{equation}
		Using the condition $0<  \lambda \leq \lambda_n $, we get
		\[\left|\frac{1}{\lambda_{n_k}}  \langle \exp_{y_{n_k}}^{-1}(x_{n_k}), \exp_{y_{n_k}}^{-1}(y) \rangle\right| \leq \frac{1}{\lambda}  |\langle \exp_{y_{n_k}}^{-1}(x_{n_k}), \exp_{y_{n_k}}^{-1}(y) \rangle|.\]
		By Lemma \ref{exp-rem}, we have
		\[\lim_{k\to \infty}\langle \exp_{y_{n_k}}^{-1}(x_{n_k}), \exp_{y_{n_k}}^{-1}(y) \rangle=0.\]
		Consequently 
		\begin{equation}\label{Npfeq2}
			\lim_{k\to \infty}\frac{1}{\lambda_{n_k}}  \langle \exp_{y_{n_k}}^{-1}(x_{n_k}), \exp_{y_{n_k}}^{-1}(y) \rangle=0.
		\end{equation}
		Using \eqref{Npfeq1} and \eqref{Npfeq2}, we obtain
		\[
		F(\bar{x},y)
		\geq
		\limsup_{k\to\infty} F(y_{n_k},y)
		\geq 0
		\text{ for all }y\in C,
		\]
		which implies that \(\bar{x} \in \ep(F,C) \).}
		Thus, $\{x_n\}$ is Fej\'er monotone with respect to
		$ \ep(F,C)$ and has a cluster point in this set.
		By Lemma~\ref{FejerLemma}, the entire sequence $\{x_n\}$ converges to $\bar x$.
	\end{proof}

\begin{remark}\label{ResEquiv}
{}{	Under Assumptions $(A1)$--$(A4)$, the Busemann and
	distance regularizations determine the same intermediate
	point $y_n$. More precisely,
	\(
	K_x=\{J_\lambda^F(x)\}\)
for all $x\in C$ and $\lambda>0$.
	Indeed, fix $x\in C$ and $\lambda>0$, and let
	$z_B=J_\lambda^F(x)$ and $z_D\in K_x$.
	By \eqref{lem-resolv-eq33}, we have
	\begin{equation}\label{ComResIneq1}
		\langle \exp_{z_B}^{-1}x,\exp_{z_B}^{-1}z_D\rangle
		\leq \lambda F(z_B,z_D).
	\end{equation}
	Similarly, by \eqref{lem-resolv-eq332}, we obtain
	\begin{equation}\label{ComResIneq2}
		\langle \exp_{z_D}^{-1}x,\exp_{z_D}^{-1}z_B\rangle
		\leq \lambda F(z_D,z_B).
	\end{equation}
	Adding \eqref{ComResIneq1} and \eqref{ComResIneq2}
	and using the monotonicity of $F$, we obtain
	\begin{align} \label{eqrem1}
		\langle \exp_{z_B}^{-1}x,\exp_{z_B}^{-1}z_D\rangle
		+\langle \exp_{z_D}^{-1}x,\exp_{z_D}^{-1}z_B\rangle\leq
		\lambda\bigl(F(z_B,z_D)+F(z_D,z_B)\bigr)
		\leq0.
	\end{align}
	 On the other hand, \eqref{triangle-ineq} gives
	\[
	\begin{aligned}
		2\langle \exp_{z_B}^{-1}x,\exp_{z_B}^{-1}z_D\rangle
		&\geq d^2(x,z_B)+d^2(z_B,z_D)-d^2(x,z_D),\\
		2\langle \exp_{z_D}^{-1}x,\exp_{z_D}^{-1}z_B\rangle
		&\geq d^2(x,z_D)+d^2(z_B,z_D)-d^2(x,z_B).
	\end{aligned}
	\]
	Adding these inequalities and using \eqref{eqrem1}, we obtain
	\[
	d^2(z_B,z_D)
	\leq
	\langle \exp_{z_B}^{-1}x,\exp_{z_B}^{-1}z_D\rangle
	+\langle \exp_{z_D}^{-1}x,\exp_{z_D}^{-1}z_B\rangle
	\leq0.
	\]
	Hence $z_B=z_D$, proving the asserted equivalence.}
	
	{}{Consequently, Algorithms~\ref{MAlg} and~\ref{MAlg2}
		provide different regularization formulations that generate
		identical exact iterates when initialized at the same point
		and supplied with the same parameter sequence. Their inexact
		implementations, however, solve differently formulated
		regularized subproblems and may therefore require different
		computational times.}
\end{remark}
	
%

	\section{Error bounds and linear convergence for strongly pseudomonotone equilibrium problems}
	
	
	In this section we establish a global error bound and demonstrate the linear
	convergence of the proposed extragradient methods (Algorithms \ref{MAlg} and
	\ref{MAlg2}) for solving strongly pseudomonotone equilibrium problems.
{Recall from Theorem \ref{SPMthm} that if the bifunction $F$ is strongly
	pseudomonotone and satisfies Assumptions $(A2)$ and $(A3)$, then the
	equilibrium problem \eqref{equib-prob} admits a unique solution.}
	
	%
	
	\subsection{Global error bounds }
	In this subsection we find error bounds for the proposed Algorithms \ref{MAlg} and \ref{MAlg2} when the  bifunction $F$ is strongly pseudomonotone.
	
	\begin{theorem}\label{Errbdd-thm}
		Let $C$ be a nonempty, closed and  geodesically convex subset of an Hadamard manifold $\M$ and let $F \colon C\times C\to \R$ be {}{a $\beta$-strongly pseudomonotone bifunction such that  Assumptions  $(A1)$--$(A4)$ hold}. Let $x^*\in \ep(F,C)$ be the unique solution to the EP and let $\{x_n\}$ be the sequence generated by Algorithm \ref{MAlg}. Then
		\[d(x_n,x^*)\leq \left(1+\frac{1}{\beta \lambda_n}\right) d(y_n,x_n) \text{ for all } n \in \N_0.\]
	\end{theorem}
	\begin{proof}
		Using \eqref{lem-resolv-eq33"} and \eqref{lem-resolv-eq33}, we have
		\begin{align}\label{GErr-eq1}
			\langle \exp_{y_n}^{-1}x_n, \exp_{y_n}^{-1}x^* \rangle \leq  	\lambda_n  F(y_n,x^*).
		\end{align}
		Since $x^*\in \ep(F,C)$ and $y_n\in C$, one obtains $F(x^*,y_n)\geq 0$. Note that $F$ is $\beta$-strongly pseudomonotone. Therefore
		\begin{equation}\label{GErr-eq2}
			F(y_n,x^*)\leq -\beta d^2(y_n,x^*).
		\end{equation}
		Using \eqref{GErr-eq1} and \eqref{GErr-eq2}, we have
		\begin{align}
			\langle \exp_{y_n}^{-1}x_n, \exp_{y_n}^{-1}x^* \rangle \leq  -\beta\lambda_n d^2(y_n,x^*).
		\end{align}
		By the Cauchy-Schwarz inequality, we have
		\begin{align*}
			\beta\lambda_n d^2(y_n,x^*) \leq \langle -\exp_{y_n}^{-1}x_n, \exp_{y_n}^{-1}x^* \rangle
			\leq \|\exp_{y_n}^{-1}x_n\| \|\exp_{y_n}^{-1}x^* \|
			= d(y_n,x_n)d(y_n,x^*),
		\end{align*}
		which implies that
		\begin{equation}\label{GErr-eq3}
			d(y_n,x^*)\leq \frac{1}{\beta\lambda_n}d(y_n,x_n).
		\end{equation}
		The triangle inequality and \eqref{GErr-eq3} imply that
		\[d(x_n,x^*)\leq d(x_n,y_n)+d(y_n,x^*)\leq d(x_n,y_n)+\frac{1}{\beta\lambda_n}d(y_n,x_n)= \left(1+\frac{1}{\beta \lambda_n}\right)d(y_n,x_n).\]
	\end{proof}
	
	\begin{theorem}
		Let $C$ be a nonempty, closed and geodesically convex subset of an Hadamard manifold $\M$ and let $F \colon C\times C\to \R$ be {}{a $\beta$-strongly pseudomonotone bifunction such that Assumptions  $(A1)$--$(A4)$ hold}. Let $x^*\in \ep(F,C)$ be the unique solution to the EP and let $\{x_n\}$ be the sequence generated by Algorithm \ref{MAlg2}. Then
		\[d(x_n,x^*)\leq \left(1+\frac{1}{\beta \lambda_n}\right) d(y_n,x_n) \text{ for all } n \in \N_0.\]
	\end{theorem}
	\begin{proof}
		Using \eqref{lem-resolv-eq332}, we have
		\begin{align}\label{GErr-eq11}
			\langle \exp_{y_n}^{-1}x_n, \exp_{y_n}^{-1}x^* \rangle \leq  	\lambda_n  F(y_n,x^*).
		\end{align}
		Since $x^*\in \ep(F,C)$ and $y_n\in C$, one obtains $F(x^*,y_n)\geq 0$. Note that $F$ is $\beta$-strongly pseudomonotone. Therefore
		\begin{equation}\label{GErr-eq21}
			F(y_n,x^*)\leq -\beta d^2(y_n,x^*).
		\end{equation}
		Using \eqref{GErr-eq11} and \eqref{GErr-eq21}, we have
		\begin{align}
			\langle \exp_{y_n}^{-1}x_n, \exp_{y_n}^{-1}x^* \rangle \leq  -\beta\lambda_n d^2(y_n,x^*).
		\end{align}
		By the Cauchy-Schwarz inequality, we have
		\begin{align*}
			\beta\lambda_n d^2(y_n,x^*) \leq \langle -\exp_{y_n}^{-1}x_n, \exp_{y_n}^{-1}x^* \rangle
			\leq \|\exp_{y_n}^{-1}x_n\| \|\exp_{y_n}^{-1}x^* \|
			= d(y_n,x_n)d(y_n,x^*),
		\end{align*}
		which implies that
		\begin{equation}\label{GErr-eq31}
			d(y_n,x^*)\leq \frac{1}{\beta\lambda_n}d(y_n,x_n).
		\end{equation}
		The triangle inequality and \eqref{GErr-eq31} imply that
		\[d(x_n,x^*)\leq d(x_n,y_n)+d(y_n,x^*)\leq d(x_n,y_n)+\frac{1}{\beta\lambda_n}d(y_n,x_n)= \left(1+\frac{1}{\beta \lambda_n}\right)d(y_n,x_n).\]
	\end{proof}
	
	\subsection{Linear convergence}

	In this subsection we establish the $R$-linear convergence rates of the proposed Algorithms \ref{MAlg} and \ref{MAlg2} when the involved bifunction $F$ is strongly pseudomonotone.
	
	\begin{definition}
		A sequence $\{x_n\}$ in an Hadamard manifold $\M$ is said to converge $R$-linearly to $x^*$ with rate $\eta \in [0,1)$ if there exists a constant $c>0$ such that $d(x_n,x^*)\leq c\eta^n$ for all $n \in \N$.
	\end{definition}

	\begin{theorem}\label{Rlinear-thm}
		Let $C$ be a nonempty, closed and geodesically  convex subset of an Hadamard manifold $\M$ and let $F \colon C\times C\to \R$ be {}{a $\beta$-strongly pseudomonotone bifunction such that  Assumptions  $(A1)$--$(A4)$ hold}. Let $\{x_n\}$ be the sequence generated by Algorithm \ref{MAlg}, where  ${0<\lambda \leq \lambda_n}$ for some $\lambda \in (0,\infty)$. Then $\{x_n\}$ converges $R$-linearly to the unique solution of EP with the rate ${r=\sqrt{1-\min\{1/2, \beta \lambda\}}}$.
	\end{theorem}

	\begin{proof}
		Under the stated assumptions, the equilibrium problem $EP(F, C)$ has a unique solution, say $x^*$. Since $x^*\in \ep(F,C)$  and $y_n\in C$, one obtains $F(x^*,y_n)\geq 0$, which, when combined with the $\beta$-strong pseudomonotonicity of $F$, implies that $F(y_n,x^*)\leq -\beta d^2(y_n,x^*)$. Letting $y=x^*$ in \eqref{lem-prof-eq1} and using the condition ${0<\lambda \leq \lambda_n}$, one has
		\begin{align}
			d^2(x_{n+1},x^*)&\leq d^2(x_n,x^*)- d^2(x_n,y_n)- d^2(x_{n+1},y_n) +2 \lambda_n F(y_n,x^*)\nonumber\\
			&\leq d^2(x_n,x^*)- d^2(x_n,y_n)- d^2(x_{n+1},y_n) -2 \beta \lambda_n d^2(y_n,x^*)\nonumber\\
			&\leq d^2(x_n,x^*)- d^2(x_n,y_n) -2 \beta \lambda_n d^2(y_n,x^*)\nonumber\\
			&\leq d^2(x_n,x^*)- d^2(x_n,y_n) -2 \beta \lambda d^2(y_n,x^*)\nonumber\\
			&{\leq d^2(x_n,x^*)- \min\{1/2, \beta \lambda\} (2d^2(x_n,y_n)+2d^2(y_n,x^*))}.  \label{Rlin-eq1}
		\end{align}
	{By the triangle inequality and the fact that $a \le b+c$ implies
			$a^2 \le 2b^2+2c^2$ for $a,b,c \ge 0$, we have
			\(
			d^2(x_n,x^*)
			\le 2d^2(x_n,y_n)+2d^2(y_n,x^*).
			\)
			Substituting this estimate into \eqref{Rlin-eq1}, we obtain}
		\begin{align*}
			d^2(x_{n+1},x^*)& {\leq d^2(x_n,x^*)- \min\{1/2, \beta \lambda\} d^2(x_n,x^*)}\nonumber\\
			& {=(1-\min\{1/2, \beta \lambda\})d^2(x_n,x^*)=r^2d^2(x_n,x^*)},
		\end{align*}
		where ${r=\sqrt{(1-\min\{1/2,\beta \lambda\})} \in [0,1)}$. Thus, we have
		\begin{equation*}
			d(x_{n+1},x^*)\leq rd(x_n,x^*) \leq \cdots \leq r^{n+1}d(x_0,x^*).
		\end{equation*}
		That is,
		\[d(x_{n+1},x^*) \leq cr^{n+1},\]
		where $c=d(x_0,x^*)$. In other words, $\{x_n\}$ is $R$-linearly convergent with rate $r$.
	\end{proof}

	\begin{theorem}
		Let $C$ be a nonempty, closed and geodesically convex subset of an Hadamard manifold $\M$ and let $F \colon C\times C\to \R$ be {}{a $\beta$-strongly pseudomonotone bifunction such that  Assumptions  $(A1)$--$(A4)$ hold}. Let $\{x_n\}$ be the sequence generated by Algorithm \ref{MAlg2}, where ${0<\lambda \leq \lambda_n}$ for some $ \lambda \in (0,\infty)$. Then $\{x_n\}$ converges $R$-linearly to the unique solution of EP with the rate ${r=\sqrt{(1-\min\{1/2, \beta \lambda\})}}$.
	\end{theorem}
	\begin{proof}
		Under the stated assumptions, the equilibrium problem $EP(F, C)$ has a unique solution, say $x^*$. Since $x^*\in \ep(F,C)$  and $y_n\in C$, one obtains $F(x^*,y_n)\geq 0$, which, when combined with the $\beta$-strong pseudomonotonicity of $F$, implies that $F(y_n,x^*)\leq -\beta d^2(y_n,x^*)$. Letting $y=x^*$ in \eqref{lem-prof-eq12} and using the condition ${0<\lambda \leq \lambda_n}$, one has
		\begin{align}
			d^2(x_{n+1},x^*)&\leq d^2(x_n,x^*)- d^2(x_n,y_n)- d^2(x_{n+1},y_n) +2 \lambda_n F(y_n,x^*)\nonumber\\
			&\leq d^2(x_n,x^*)- d^2(x_n,y_n)- d^2(x_{n+1},y_n) -2 \beta \lambda_n d^2(y_n,x^*)\nonumber\\
			&\leq d^2(x_n,x^*)- d^2(x_n,y_n) -2 \beta \lambda_n d^2(y_n,x^*)\nonumber\\
			&\leq d^2(x_n,x^*)- d^2(x_n,y_n) -2 \beta \lambda d^2(y_n,x^*)\nonumber\\
			&{\leq d^2(x_n,x^*)- \min\{1/2, \beta \lambda\} (2d^2(x_n,y_n)+2d^2(y_n,x^*))}.  \label{Rlin-eq11}
		\end{align}
		{By the triangle inequality and the fact that $a \le b+c$ implies
			$a^2 \le 2b^2+2c^2$ for $a,b,c \ge 0$, we have
			\(
			d^2(x_n,x^*)
			\le 2d^2(x_n,y_n)+2d^2(y_n,x^*).
			\)
			Substituting this estimate into \eqref{Rlin-eq11}, we obtain}
		\begin{align*}
			d^2(x_{n+1},x^*)& {\leq d^2(x_n,x^*)- \min\{1/2, \beta \lambda\} d^2(x_n,x^*)}\nonumber\\
			& {=(1-\min\{1/2, \beta \lambda\})d^2(x_n,x^*)=r^2d^2(x_n,x^*)},
		\end{align*}
		where ${r=\sqrt{(1-\min\{1/2,\beta \lambda\})} \in [0,1)}$. Thus, we have
		\begin{equation*}
			d(x_{n+1},x^*)\leq rd(x_n,x^*) \leq \cdots \leq r^{n+1}d(x_0,x^*).
		\end{equation*}
		That is,
		\[d(x_{n+1},x^*) \leq cr^{n+1},\]
		where $c=d(x_0,x^*)$. In other words, $\{x_n\}$ is $R$-linearly convergent with rate $r$.
	\end{proof}
	

	\newpage
	\section{Numerical Experiments}
	
	In this section we present  numerical examples to illustrate the computational performance of our proposed algorithms and compare them with existing ones. All experiments were implemented in MATLAB R2026a and executed on a MacBook Air equipped with an Apple M5 processor and 16 GB of RAM. The MATLAB codes used for the numerical experiments are publicly available at
	\url{https://github.com/shikhers04/Regularized-Extragradient-Hadamard}.

	We begin our numerical experiments on a flat Hadamard manifold, which allows us to assess the effectiveness and performance of the proposed algorithms.
	
	Let $\R^N_{++}=\{x=(x_1,\ldots, x_N)\in \R^N \colon x_i>0, ~ i=1,\ldots,N\}$, and let $\M=(\R^N_{++},\langle \cdot,\cdot \rangle)$ be a Riemannian manifold with the
	Riemannian metric defined by
	\begin{equation}\label{NumRM}
		\langle u,v \rangle =uG(x)v^T, \quad x\in \R^N_{++}, u,v \in T_x\R^N_{++}=\R^N,
	\end{equation}
	where $G(x)$ is a diagonal matrix defined by $G(x)= \text{diag}(x_1^{-2},\ldots, x_N^{-2})$.
	The Riemannian distance $d \colon \M\times \M\to [0,\infty)$ is defined by
	$$d(x,y)= \left(\sum_{i=1}^N \ln^2 \frac{x_i}{y_i}\right)^{1/2} \text{ for all } x, y \in  \M.$$
	The sectional curvature of this Riemannian manifold $\M$ is $0$. Thus $M=(\R^N_{++}, \langle \cdot, \cdot \rangle) $ is an Hadamard manifold; see \cite{Shikher2024NA}.
	
	Let $x,z\in \M$. Then the Busemann function $b_{\gamma_{z,x}} \colon \M \to  \R$ is given by 
	\[b_{\gamma_{z,x}}(y)=\frac{ \sum_{i=1}^N\ln\left(\frac{x_i}{z_i}\right)\ln\left(\frac{z_i}{y_i}\right)}{\left(\sum_{i=1}^N \ln^2 \frac{x_i}{z_i}\right)^{1/2}} \text{ for all } y \in  \M.\]

\begin{example}\rm\label{Exm1}
{}{	Let $\M=\mathbb{R}_{++}^{3}$ be equipped with the Riemannian
	metric defined by \eqref{NumRM}, and let $C=\M$. Set
	\[
	B=
	\begin{pmatrix}
		2 & -1 & 0\\
		-1 & 2 & -1\\
		0 & -1 & 2
	\end{pmatrix}.
	\] and 
	define $F\colon \M\times\M\to\mathbb{R}$ by
	\[
	F(x,y)
	=
	\bigl(B\log x+\sinh(\log x)\bigr)^{\mathsf T}
	(\log y-\log x) \text{ for all } x,y\in\M,
	\]
	where $\log$ and $\sinh$ are applied componentwise.
	Recall that
	\(
	d(x,y)=\|\log x-\log y\|_2.
	\)}
	We verify Assumptions $(A1)$--$(A4)$.
	
	\begin{enumerate}
		\item[$(A1)$]
		Clearly, $F(x,x)=0$. The smallest eigenvalue of $B$ is
		$2-\sqrt{2}$. Moreover,
		\((\sinh a-\sinh b)(a-b)\geq (a-b)^2\) for all  \(a,b\in\mathbb{R}.\)
		Consequently,
		\begin{align*}
			&F(x,y)+F(y,x)\\
			&=-(\log x-\log y)^{\mathsf T}
			B(\log x-\log y)
			-\sum_{i=1}^{3}
			\bigl(\sinh(\log x_i)-\sinh(\log y_i)\bigr)
			(\log x_i-\log y_i)\\
			&\leq-\bigl((2-\sqrt{2})+1\bigr)
			\|\log x-\log y\|_2^2=-(3-\sqrt{2})d^2(x,y).
		\end{align*}
		Thus $F$ is strongly monotone, and hence monotone.

		\item[$(A2)$]
		The geodesic joining $y$ and $z$ is \(\gamma(t)=
		\exp\bigl((1-t)\log y+t\log z\bigr)\), \(t\in[0,1]\).
		For every fixed $x\in\M$,
		\[
		F(x,\gamma(t))
		=(1-t)F(x,y)+tF(x,z).
		\]
		Hence $F(x,\cdot)$ is geodesically affine, and therefore
		geodesically convex. It is also smooth and thus
		subdifferentiable.
		
		\item[$(A3)$]
		For every fixed $y\in\M$, the function $F(\cdot,y)$
		is continuous, and hence upper semicontinuous.
		
		\item[$(A4)$]
		Let $x^*=(1,1,1)^{\mathsf T}$. Since $F(x^*,x)=0$,
		the inequality in $(A1)$ gives
		\[
		F(x,x^*)\leq -(3-\sqrt{2})d^2(x,x^*)\leq 0
	\text{ for all }x\in\M.
		\]
		In particular, the condition in $(A4)$ holds for every
		sequence escaping to infinity.
	\end{enumerate}
	
	Since $F(x^*,y)=0$ for every $y\in\M$, and $F$ is
	strongly monotone, the equilibrium problem has the unique
	solution
	\(
	\ep(F,\M)=\{(1,1,1)\}.
	\)
\end{example}

We report the numerical performance of Algorithm~\ref{MAlg},
denoted by REMB, and Algorithm~\ref{MAlg2}, denoted by REMD,
for the equilibrium problem~\eqref{equib-prob} with the
bifunction defined in Example~\ref{Exm1}. {}{As established in Remark~\ref{ResEquiv},
the two regularization formulations generate identical
exact iterates. We compare their inexact implementations,
which solve differently formulated regularized subproblems
and may therefore require different computational times.}

To solve the subproblem arising in the first step of REMB
and REMD, we employ Algorithm~2 of \cite{Khammahawong2020}, with inner step sizes $\eta_k=1/(k+1)$. For each outer iteration
$n$, the inner sequence $\{z_k\}$ starts at $z_0=x_n$
and terminates when $d(z_k,z_{k-1})\leq10^{-6}$.

We set $\lambda_n=\lambda$ and consider ten values of
$\lambda$ generated using the MATLAB command
\texttt{3*linspace(0.01,0.10,10)}.
Figures~\ref{Fig1} and \ref{Fig2} show the convergence
curves for REMB and REMD, respectively, each with initial
points $x_0=(1,2,3)$ and
$x_0=(5,5,5)$. We also perform 30 independent
trials with initial points generated using the  MATLAB command
\texttt{randi([5,20],3,1)}. The outer stopping criterion is
$Er(n)=d(x_{n+1},x^*)\leq10^{-8}$.

{}{Tables~\ref{table1} and~\ref{table2} summarize the performance
of REMB and REMD over 30 trials with randomly generated
initial points. For both implementations, the mean outer
iteration count, mean total inner iteration count, and
mean computational time decrease as $\lambda$ increases.
Among the tested values, $\lambda=0.30$ yields the lowest
means for all three quantities. Outer iteration counts
vary little across the sampled initial points, whereas
the standard deviations of computational times do not
decrease monotonically.}

{}{Figures~\ref{FigBox1} and~\ref{FigBox2} present box plots
of outer iteration counts and computational times for
REMB and REMD, respectively, over 30 trials for each
$\lambda$. The central line denotes the median, the box
spans the 25th--75th percentiles, and the whiskers extend
to the most extreme nonoutlying observations. Outliers
are marked by `+'.}

\begin{table}[htbp]
	\centering
		\begin{tabular}{cllccc}
\hline
$\lambda$ & Mean Out. Itr. & Std Dev & Mean Inn. Iter. & Mean Time (s) & Std Dev 
\\ \hline
0.03 & 413.57 & 1.33 & 2278.27 & 0.004048 & 0.000834 \\
0.06 & 208.57 & 0.73 & 1271.63 & 0.002113 & 0.000031 \\
0.09 & 140.37 & 0.61 & 909.53 & 0.001491 & 0.000024 \\
0.12 & 106.23 & 0.57 & 725.50 & 0.001193 & 0.000020 \\
0.15 & 85.83 & 0.38 & 620.30 & 0.001012 & 0.000027 \\
0.18 & 72.13 & 0.35 & 545.60 & 0.000894 & 0.000022 \\
0.21 & 62.47 & 0.51 & 493.83 & 0.000802 & 0.000019 \\
0.24 & 55.20 & 0.41 & 464.23 & 0.000745 & 0.000027 \\
0.27 & 49.83 & 0.38 & 450.07 & 0.000718 & 0.000031 \\
0.30 & 44.97 & 0.18 & 444.70 & 0.000710 & 0.000041 \\
\hline
	\end{tabular}
	\caption{Numerical results for Alg. REMB with different values of $\lambda$ and randomly generated initial points}
	\label{table1}
\end{table}
\begin{table}[htbp]
	\centering
		\begin{tabular}{cllccc}
			$\lambda$ & Mean Out. Itr. & Std Dev & Mean Inn. Iter. & Mean Time (s) & Std Dev 
			\\ \hline
			0.03 & 413.27 & 1.26 & 7288.60 & 0.012508 & 0.005055 \\
			0.06 & 208.30 & 0.70 & 4051.37 & 0.006509 & 0.000183 \\
			0.09 & 140.03 & 0.56 & 2864.87 & 0.004474 & 0.000067 \\
			0.12 & 106.03 & 0.41 & 2244.83 & 0.003503 & 0.000063 \\
			0.15 & 85.60 & 0.50 & 1875.07 & 0.002924 & 0.000040 \\
			0.18 & 71.93 & 0.25 & 1606.80 & 0.002515 & 0.000056 \\
			0.21 & 62.13 & 0.35 & 1413.57 & 0.002200 & 0.000033 \\
			0.24 & 54.97 & 0.18 & 1266.97 & 0.001965 & 0.000024 \\
			0.27 & 49.30 & 0.47 & 1150.60 & 0.001794 & 0.000039 \\
			0.30 & 44.93 & 0.25 & 1059.67 & 0.001650 & 0.000031 \\
	\end{tabular}
	\caption{Numerical results for Alg. REMD with different values of $\lambda$ and randomly generated initial points}
	\label{table2}
\end{table}

\begin{figure}[htbp]
	\centering
	\begin{subfigure}[c]{0.48\textwidth}
		\centering
		\includegraphics[width=\textwidth]{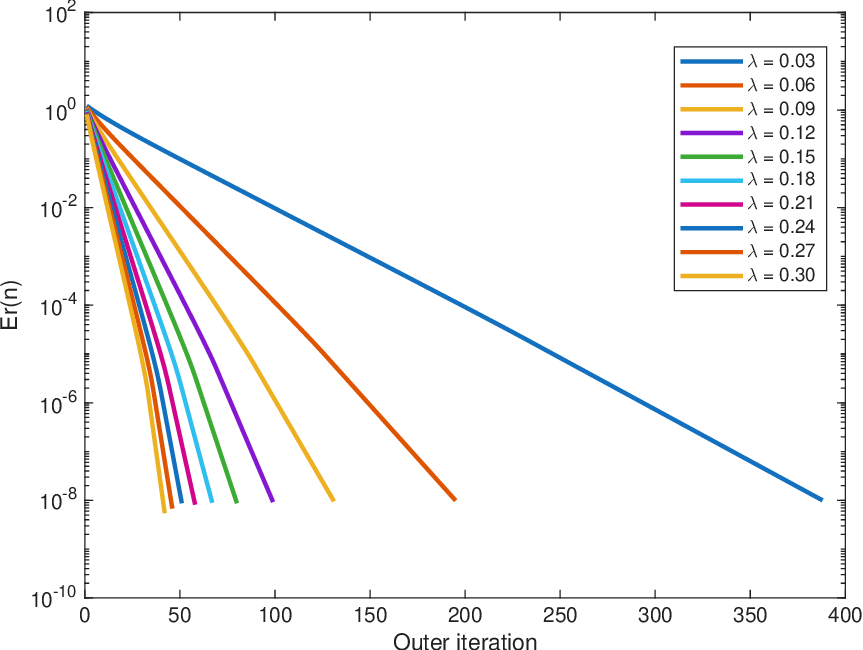}
		\caption{$x_0=(1,2,3)$}
		\label{fig:sub1}
	\end{subfigure}%
	\hfill
	\begin{subfigure}[c]{0.48\textwidth}
		\centering
		\includegraphics[width=\textwidth]{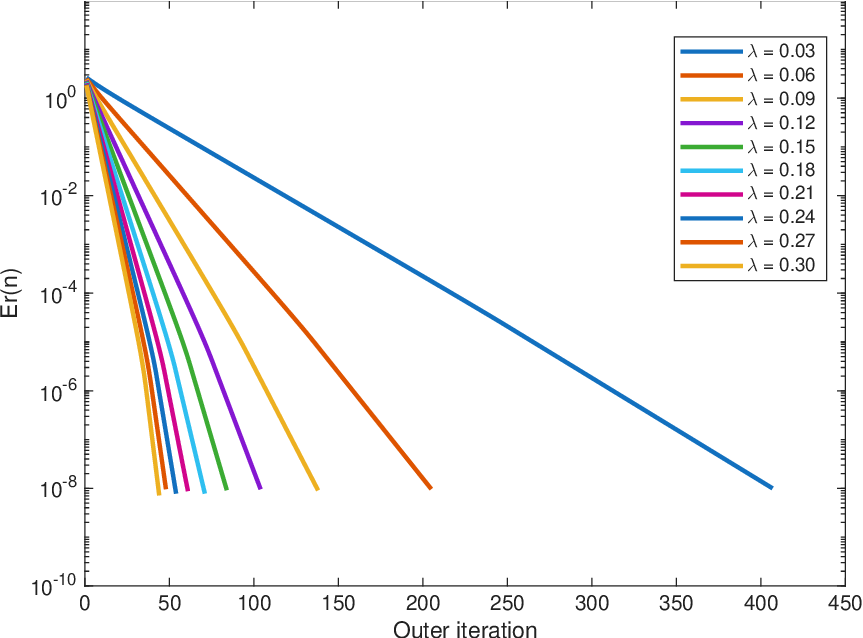}
		\caption{$x_0=(5,5,5)$}
		\label{fig:sub2}
	\end{subfigure}
	\caption{Convergence behavior of Alg. REMB for different values of $\lambda$ with initial points $x_0 = (1, 2, 3)$ and $x_0 = (5, 5, 5)$ in Example \ref{Exm1}}
	\label{Fig1}
\end{figure}

\begin{figure}[htbp]
	\centering
	\begin{subfigure}[b]{0.48\textwidth}
		\centering
		\includegraphics[width=\textwidth]{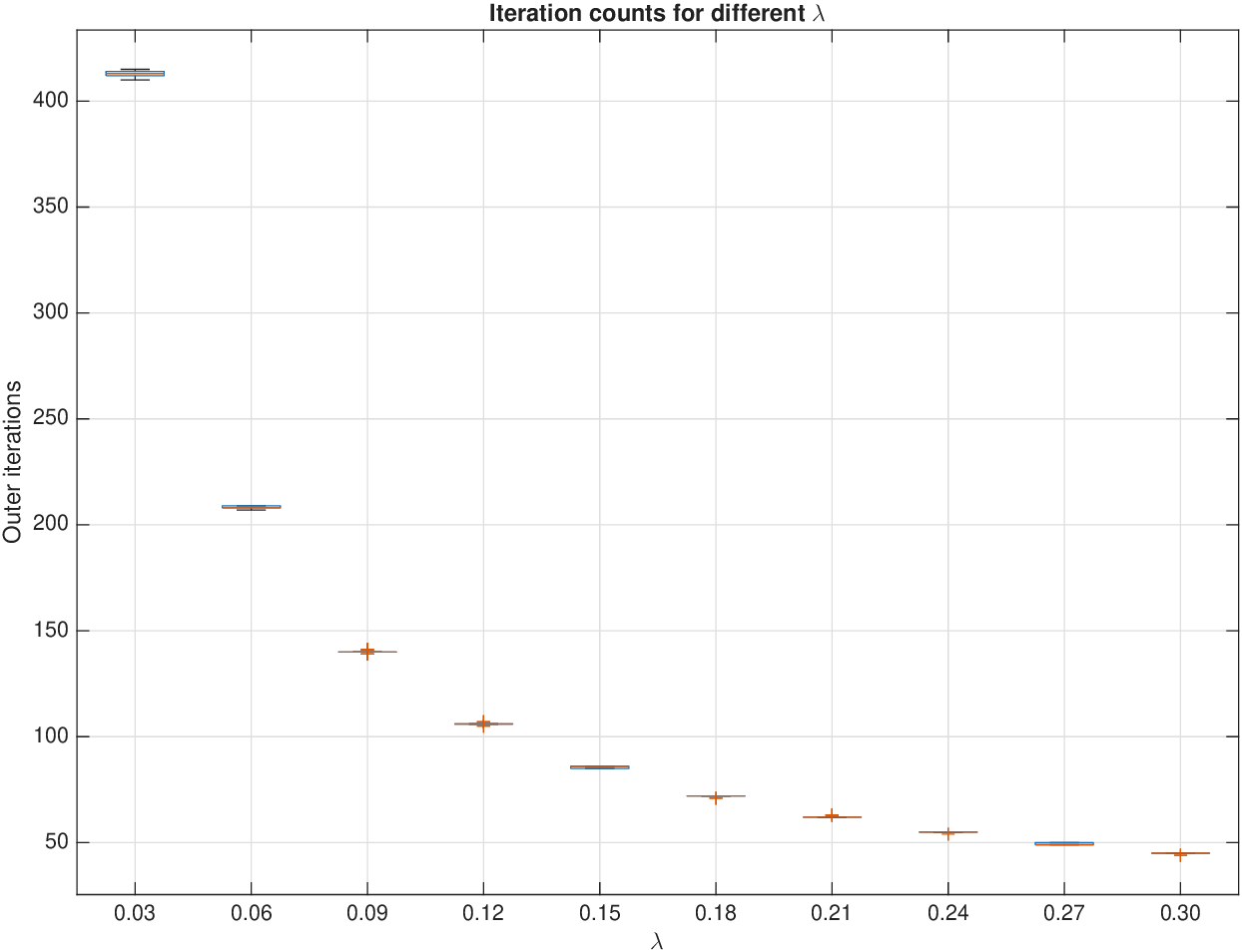}
		\caption{Box plot of iterations}
		\label{fig:BoxBuseItr}
	\end{subfigure}%
	\hfill
	\begin{subfigure}[b]{0.48\textwidth}
		\centering
		\includegraphics[width=\textwidth]{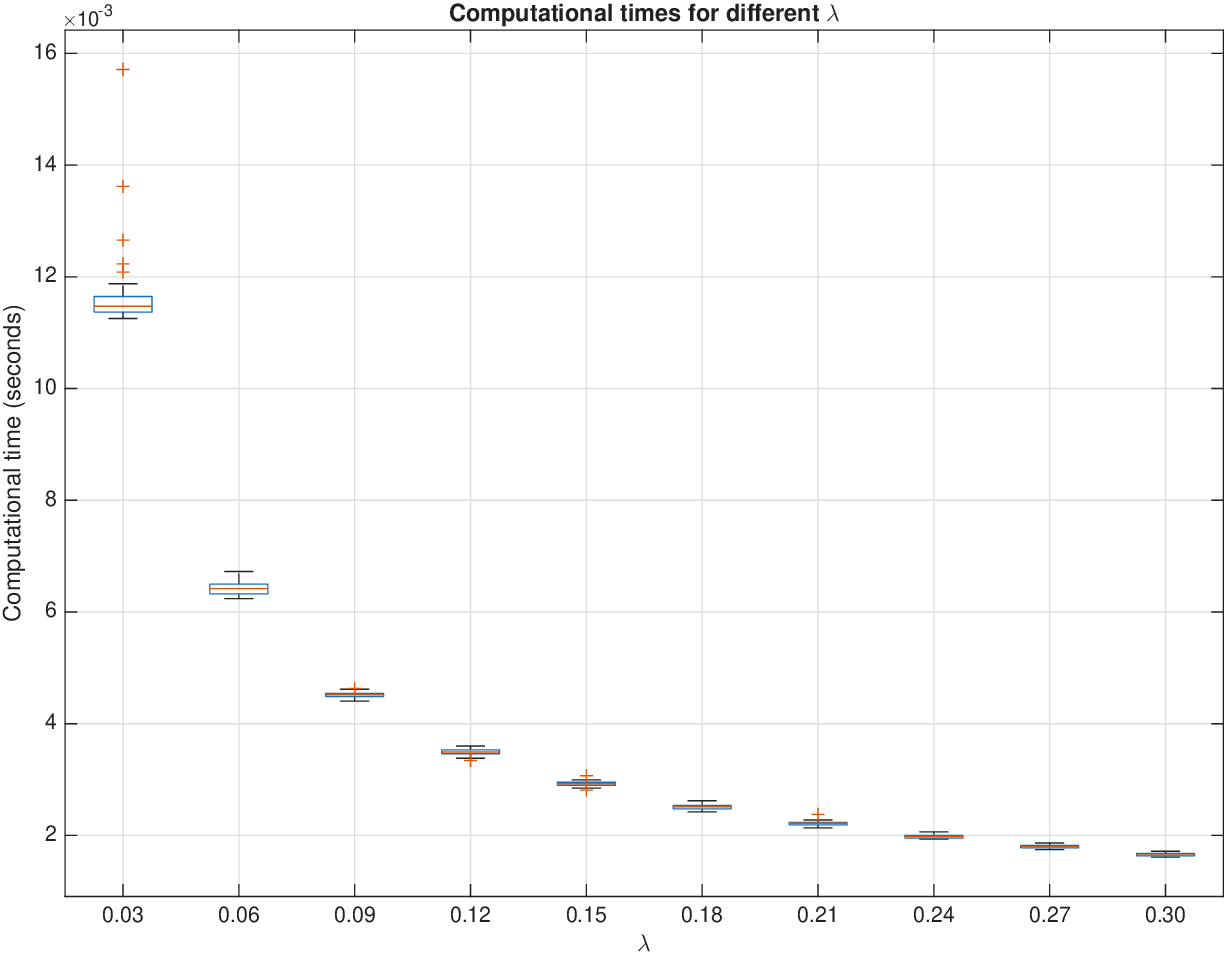}
		\caption{Box plot of computational time}
		\label{fig:BoxBuseCpu}
	\end{subfigure}
	\caption{Box plots of Alg. REMB for various values of $\lambda$ and randomly generated initial points in Example \ref{Exm1}}
	\label{FigBox1}
\end{figure}

\begin{figure}[htbp]
	\centering
	\begin{subfigure}[c]{0.48\textwidth}
		\centering
		\includegraphics[width=\textwidth]{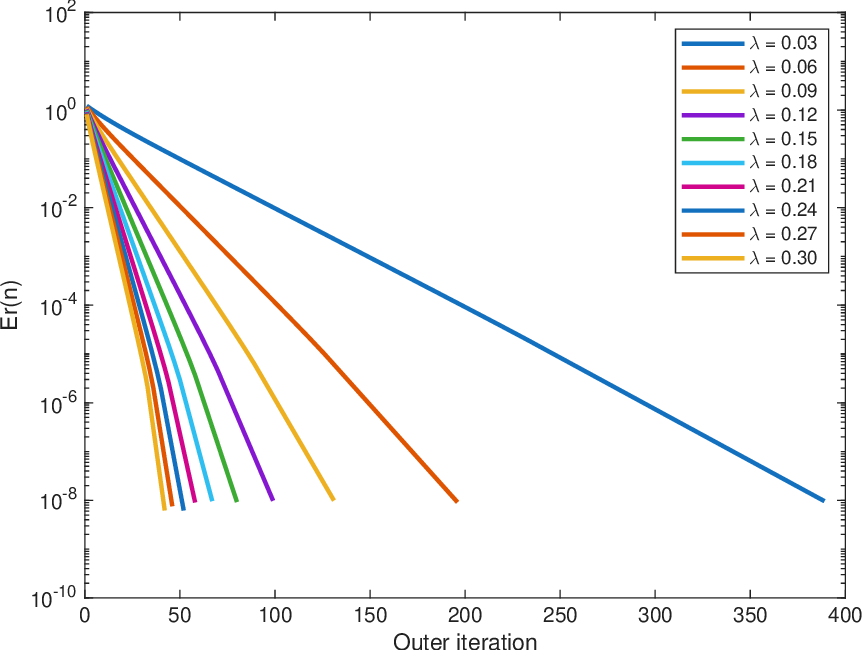}
		\caption{$x_0=(1,2,3)$}
		\label{fig:LDmetric123}
	\end{subfigure}
	\hfill
	\begin{subfigure}[c]{0.48\textwidth}
		\centering
		\includegraphics[width=\textwidth]{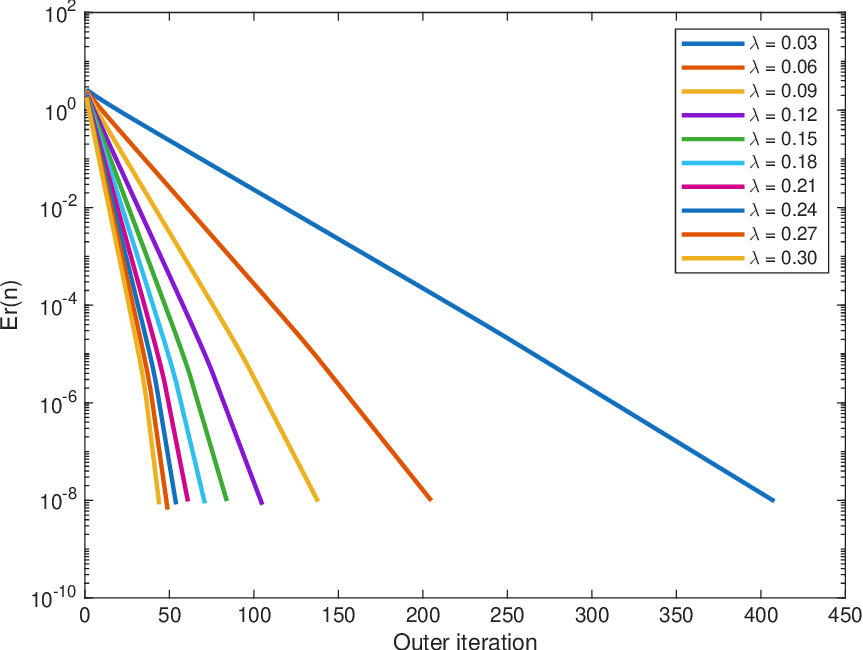}
		\caption{$x_0=(5,5,5)$}
		\label{fig:LDmetric555}
	\end{subfigure}
	\caption{Convergence behavior of Alg. REMD for different values of $\lambda$ with initial points $x_0 = (1, 2, 3)$ and $x_0 = (5, 5, 5)$ in  Example \ref{Exm1}}
	\label{Fig2}
\end{figure}

\begin{figure}[htbp]
	\centering
	\begin{subfigure}[b]{0.48\textwidth}
		\centering
		\includegraphics[width=\textwidth]{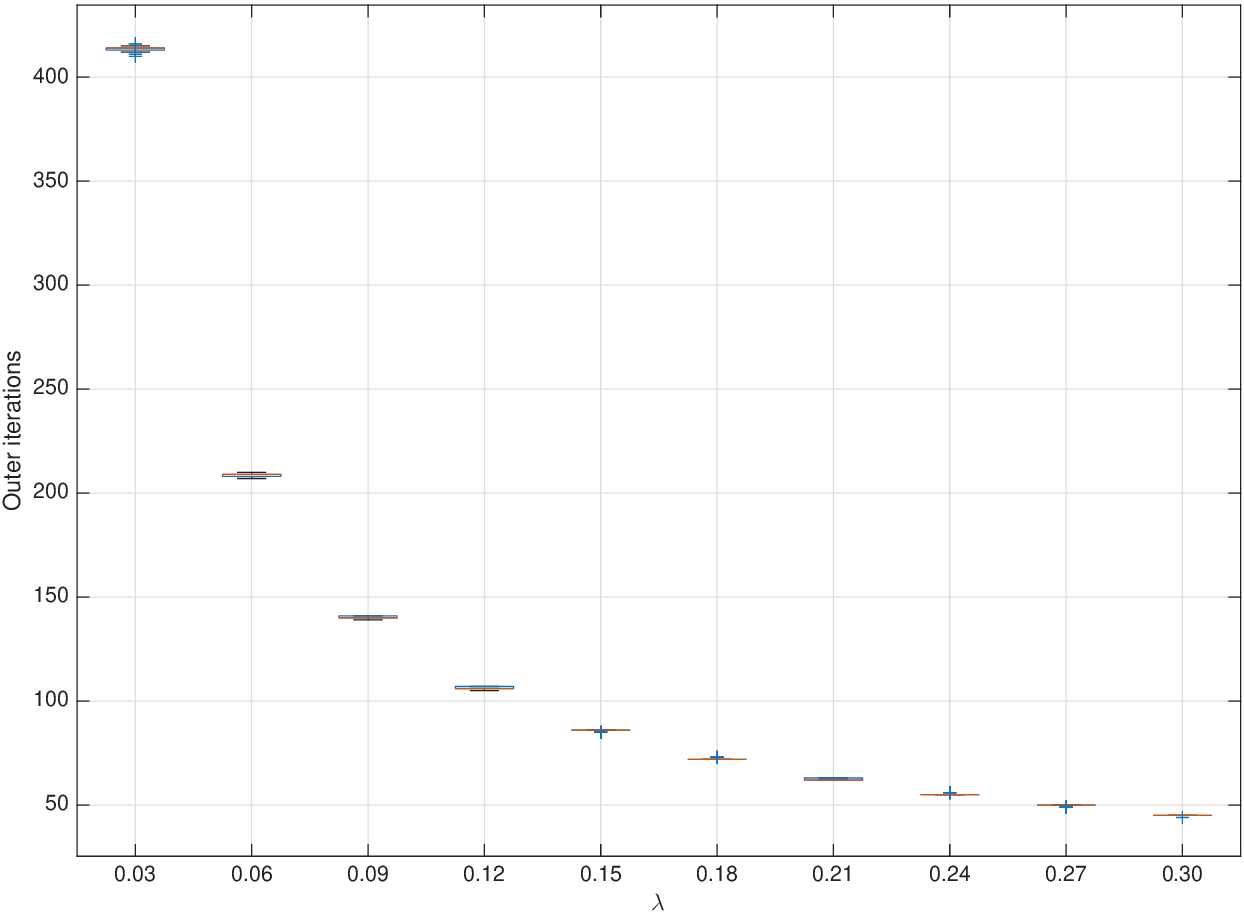}
		\caption{Box plot of iterations}
		\label{fig:BoxDmetricItr}
	\end{subfigure}%
	\hfill
	\begin{subfigure}[b]{0.48\textwidth}
		\centering
		\includegraphics[width=\textwidth]{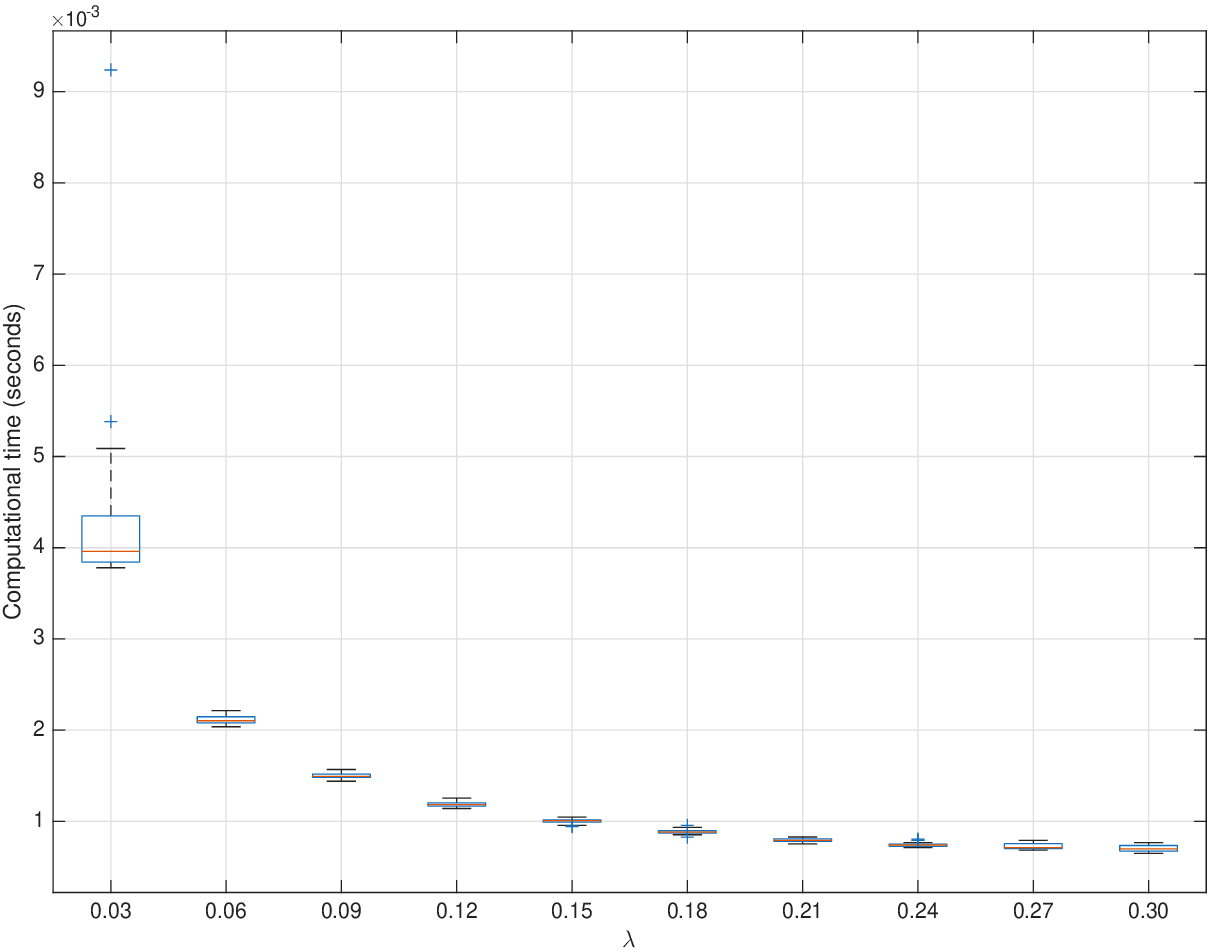}
		\caption{Box plot of computational time}
		\label{fig:BoxDmetricCpu}
	\end{subfigure}
	\caption{Box plots of Alg. REMD for various values of $\lambda$ and randomly generated initial points in Example \ref{Exm1}}
	\label{FigBox2}
\end{figure}

In the following example, we illustrate the $R$-linear convergence of Alg. REMB and Alg. REMD for a strongly pseudomonotone equilibrium problem.
We compare the observed contraction ratios with the exact convergence factor and the theoretical rate bound.
\begin{example}\rm \label{Ex22}
	
	Let $\M$ be the Riemannian manifold with the Riemannian metric defined by \eqref{NumRM}. Let $u=(\ln x_1,\ldots,\ln x_N)$ and
	$v=(\ln y_1,\ldots,\ln y_N)$ and  let $F \colon \M \times \M\to \R$ be defined by
	\[F(x,y)=\langle u, v-u  \rangle=\left(\ln x_1\ln\left(\frac{y_1}{x_1}\right)+ \cdots + \ln x_N\ln\left(\frac{y_N}{x_N}\right)\right) \text{ for all } x,y\in \M.\]
	

Clearly, $F(x,x)=0$ for all $x\in\M$.
Moreover, $F$ is monotone and, in fact,
$1$-strongly monotone. Indeed, for any
$x,y\in\M$, we have
\[
\begin{aligned}
	F(x,y)+F(y,x)
	=\langle u,v-u\rangle+\langle v,u-v\rangle
	=-\|u-v\|^2
	=-d^2(x,y).
\end{aligned}
\]
Furthermore, if $F(x,y)\geq0$, then
\[
F(y,x)
=-d^2(x,y)-F(x,y)
\leq-d^2(x,y).
\]
{}{Thus, $F$ is also $\beta$-strongly pseudomonotone
with $\beta=1$. Its unique equilibrium point is
$x^*=(1,\ldots,1)$.}

Next, we verify conditions $(A2)$--$(A4)$.
For any $x,y,z\in\M$ and $t\in[0,1]$, we have
\[
F(x,\gamma_{y,z}(t))
=(1-t)F(x,y)+tF(x,z),
\]
where $\gamma_{y,z}$ denotes the geodesic
joining $y$ and $z$.
Thus, $F(x,\cdot)$ is geodesically convex and
subdifferentiable, which verifies $(A2)$.
Furthermore, $F(\cdot,y)$ is continuous
for every $y\in\M$, so $(A3)$ holds.
Finally, choosing $x^*=(1,\ldots,1)$,
we obtain
\[
F(z,x^*)=-\|\log z\|^2
=-d^2(z,x^*)\leq0
\quad\text{for all }z\in\M.
\]
Hence, condition $(A4)$ is also satisfied.

Consequently, the resolvent
$J_\lambda^F:\M\to\M$ is given by
\[
J_\lambda^F(x)
=
\left(
x_1^{\frac{1}{1+\lambda}},
\ldots,
x_N^{\frac{1}{1+\lambda}}
\right)
\quad\text{for all }x\in\M.
\]
For each $x\in\M$, the set $K_x$ is given by
\[
K_x=
\left\{
\left(
x_1^{\frac{1}{1+\lambda}},
\ldots,
x_N^{\frac{1}{1+\lambda}}
\right)
\right\}.
\]
	
{}{	Consider a constant parameter sequence $\lambda_n=\lambda>0$.
	Since $K_x=\{J_\lambda^F(x)\}$, the first step of both
	Alg. REMB and Alg. REMD yields
	\[
	y_{n,i}=x_{n,i}^{1/(1+\lambda)},
	\qquad i=1,\ldots,N.
	\]
	Writing $u_n=\log x_n$,  $v_n=\log y_n$, and
	$w=\log z$, 
	where the logarithm is applied componentwise, and using the fact that
	$\log:\M\to\mathbb{R}^N$ is a bijection, we see that the second
	minimization step in both algorithms becomes
	\[
	u_{n+1}
	=
	\operatorname*{argmin}_{w\in\mathbb{R}^N}
	\left\{
	\langle v_n,w-v_n\rangle
	+\frac{1}{2\lambda}\|w-u_n\|_2^2
	\right\}.
	\]
	Its optimality condition gives
	\[
	u_{n+1}
	=
	u_n-\lambda v_n
	=
	\frac{u_n}{1+\lambda}
	=
	v_n.
	\]
	In particular, $x_{n+1}=y_n$, and Alg. REMB and Alg. REMD generate
	the same sequence when initialized at the same point
	and supplied with the same parameter sequence.}
	
{}{	Define the true error by \(e_n:=d(x_n,x^*)\). Then
	\(
	e_n=\|\log x_n-\log x^*\|_2=\|\log x_n\|_2.
	\)
	Since
	\[
	e_{n+1}=\frac{1}{1+\lambda}e_n,
	\]
	we obtain
	\[
	e_n=\left(\frac{1}{1+\lambda}\right)^n e_0.
	\]
	Thus, the sequence converges $R$-linearly to $x^*$,
	with convergence factor $q=(1+\lambda)^{-1}\in(0,1)$.}
	
{}{ For $\beta=1$, the theoretical convergence factor is
	\[
	r:=\sqrt{1-\min\{1/2,\lambda\}}.
	\]
	Since $q\leq r<1$, the exact error sequence satisfies
	\[
	e_n=e_0q^n\leq e_0r^n,
	\qquad n\geq0.
	\]
	This example therefore provides an explicit benchmark
	for assessing the theoretical rate estimate.}
	
\end{example}

{}{To verify the convergence rate numerically, we consider
$N=100$ and generate an initial point $x_0$ whose components
are randomly chosen from $\{5,\ldots,20\}$. The same initial
point is used for both algorithms and for all parameter
values. We perform $150$ iterations for each
$\lambda\in\{0.06,0.12,0.18,0.30\}$.}

{}{At each iteration, we record the true error
\(
e_n:=d(x_n,x^*)
\)
and the observed contraction ratio
\(
\rho_n:=\frac{e_{n+1}}{e_n}\), \(
 n=0,\ldots,149.
\)
For the chosen initial point, the errors are positive
in exact arithmetic. We summarize the observed ratios
by their arithmetic mean
\(
\bar{\rho}:=\frac{1}{150}\sum_{n=0}^{149}\rho_n.
\)
For this example, the exact contraction factor is
\(
q=\frac{1}{1+\lambda},
\)
whereas the general convergence theorem provides the bound
\(
r=\sqrt{1-\min\{1/2,\lambda\}},
\)
since $\beta=1$. In exact arithmetic,
\(
e_n=e_0q^n\leq e_0r^n
\quad\text{and}\quad
\rho_n=q.
\)
We therefore compare the observed mean contraction factor
$\bar{\rho}$ with $q$ and $r$, and plot the true error
together with the theoretical envelope $e_0r^n$
on a semilogarithmic scale.}

{}{Since Alg. REMB and Alg. REMD generate identical exact
iterates for this example when initialized at the same
point with the same parameter sequence, a single error
curve for each value of $\lambda$ represents both methods.}

\begin{table}[htbp]
	\centering
	\begin{tabular}{cccc}
		\hline
		$\lambda$
		& Observed $\bar{\rho}$
		& Exact factor $q$
		& Theoretical bound $r$ \\
		\hline
		$0.06$ & $0.943396$ & $0.943396$ & $0.969536$ \\
		$0.12$ & $0.892857$ & $0.892857$ & $0.938083$ \\
		$0.18$ & $0.847458$ & $0.847458$ & $0.905539$ \\
		$0.30$ & $0.769231$ & $0.769231$ & $0.836660$ \\
		\hline
	\end{tabular}
	\caption{Observed mean contraction factors, exact
		contraction factors, and theoretical rate bounds.}
	\label{tab:linear-rate-factors}
\end{table}

\begin{figure}[htbp]
	\centering
	
	\begin{subfigure}[t]{0.48\textwidth}
		\centering
		\includegraphics[width=\linewidth]{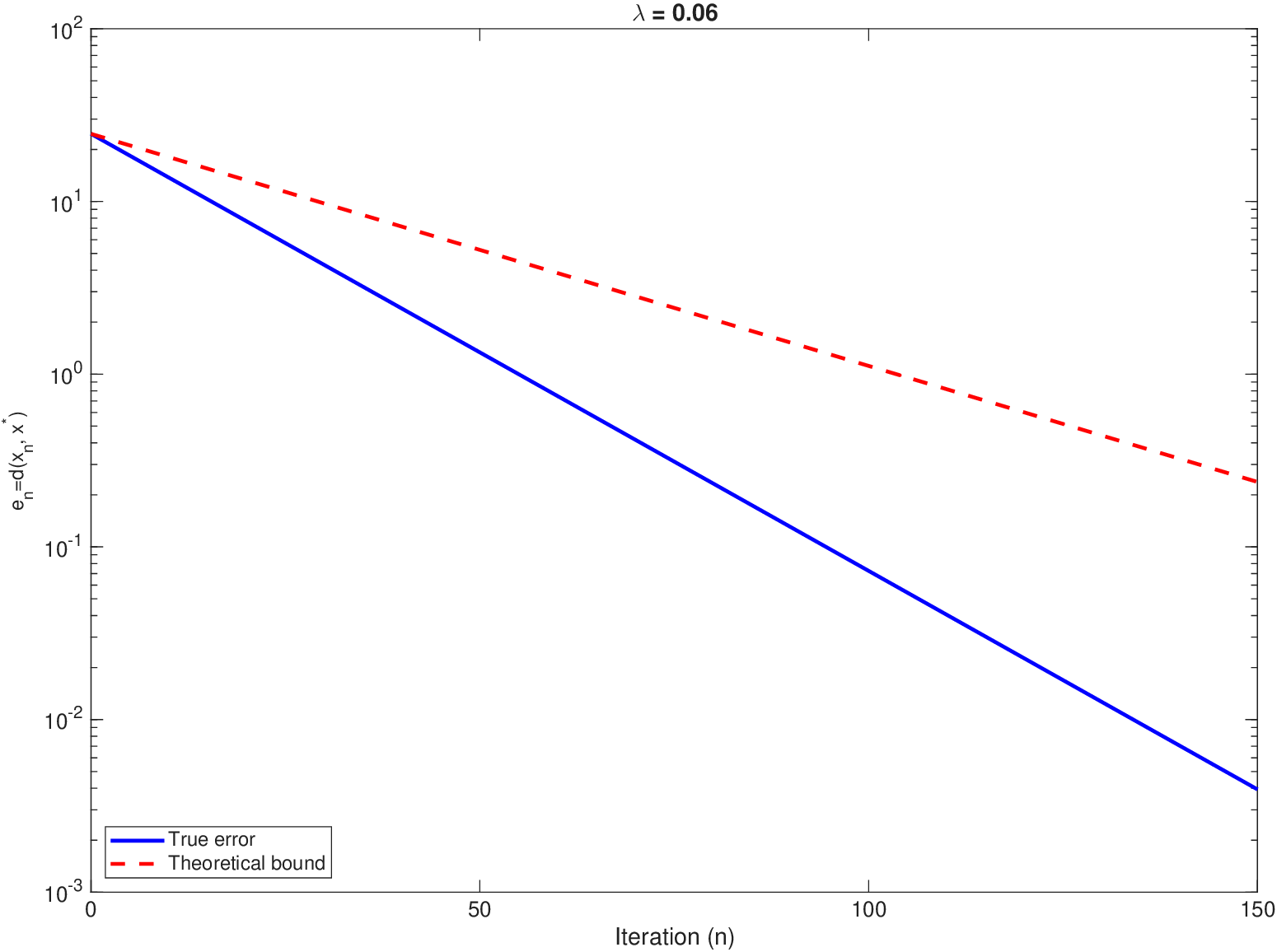}
		\caption{$\lambda=0.06$}
		\label{fig:rate-006}
	\end{subfigure}
	\hfill
	\begin{subfigure}[t]{0.48\textwidth}
		\centering
		\includegraphics[width=\linewidth]{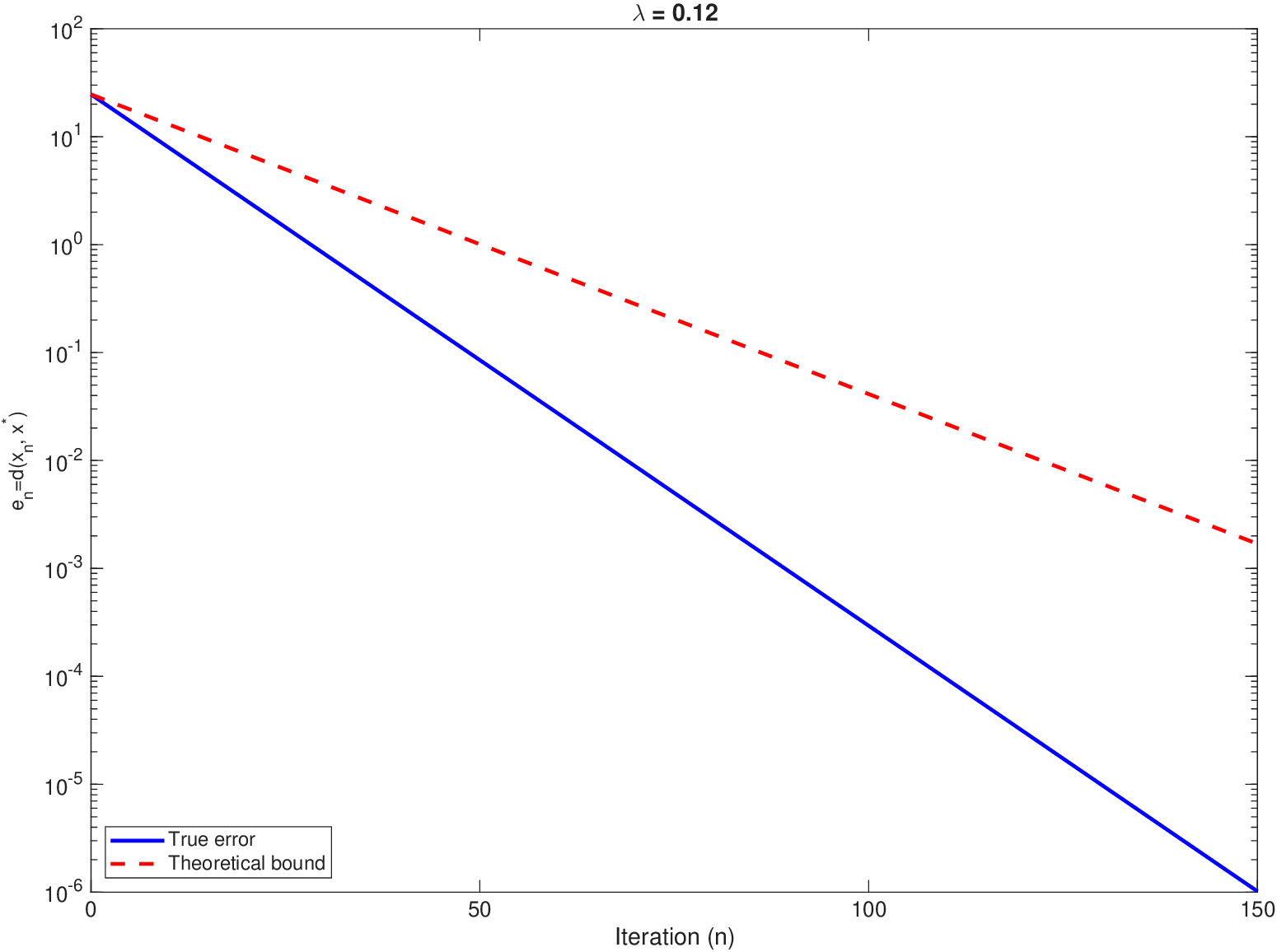}
		\caption{$\lambda=0.12$}
		\label{fig:rate-012}
	\end{subfigure}
	
	\medskip
	
	\begin{subfigure}[t]{0.48\textwidth}
		\centering
		\includegraphics[width=\linewidth]{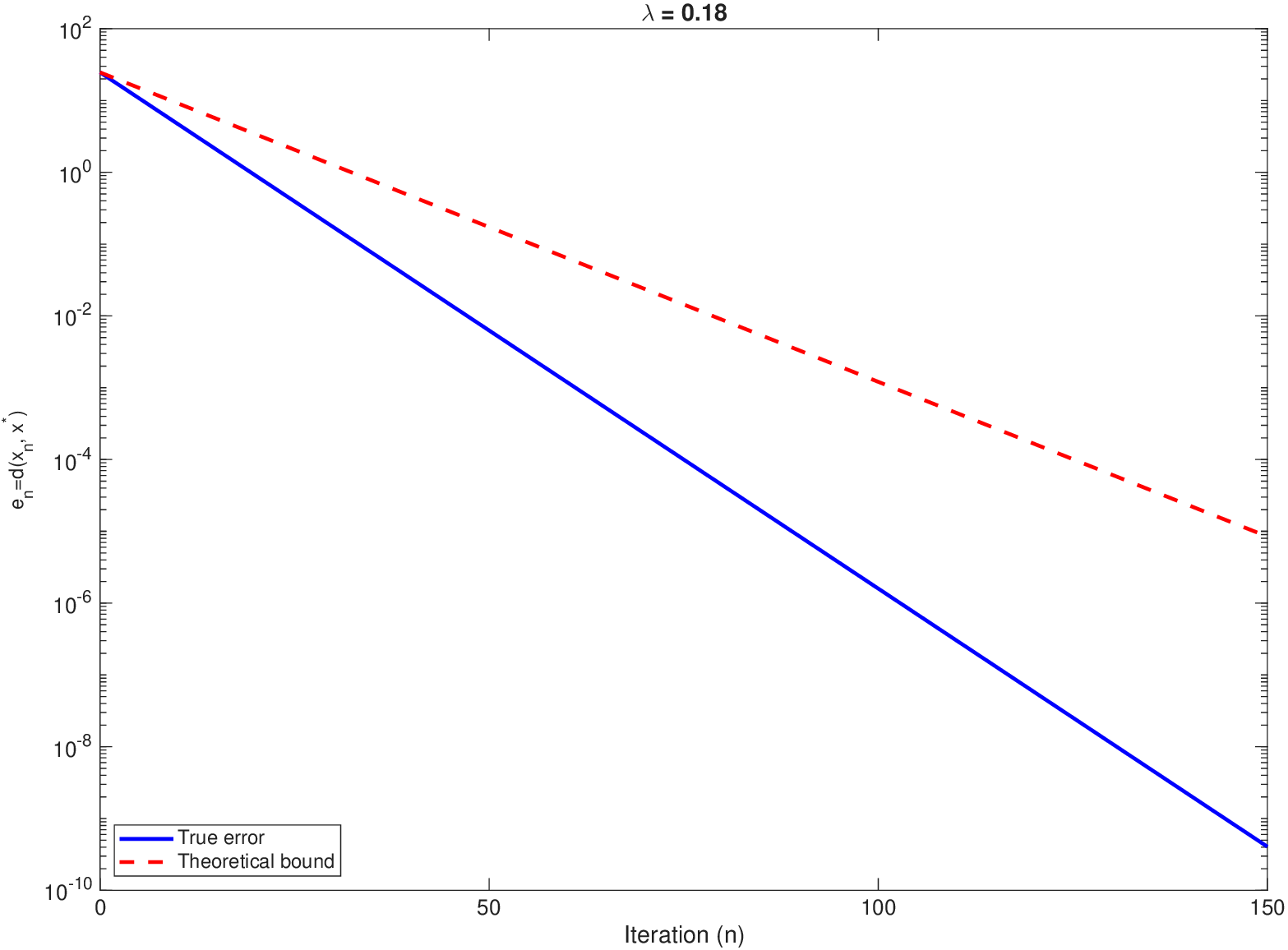}
		\caption{$\lambda=0.18$}
		\label{fig:rate-018}
	\end{subfigure}
	\hfill
	\begin{subfigure}[t]{0.48\textwidth}
		\centering
		\includegraphics[width=\linewidth]{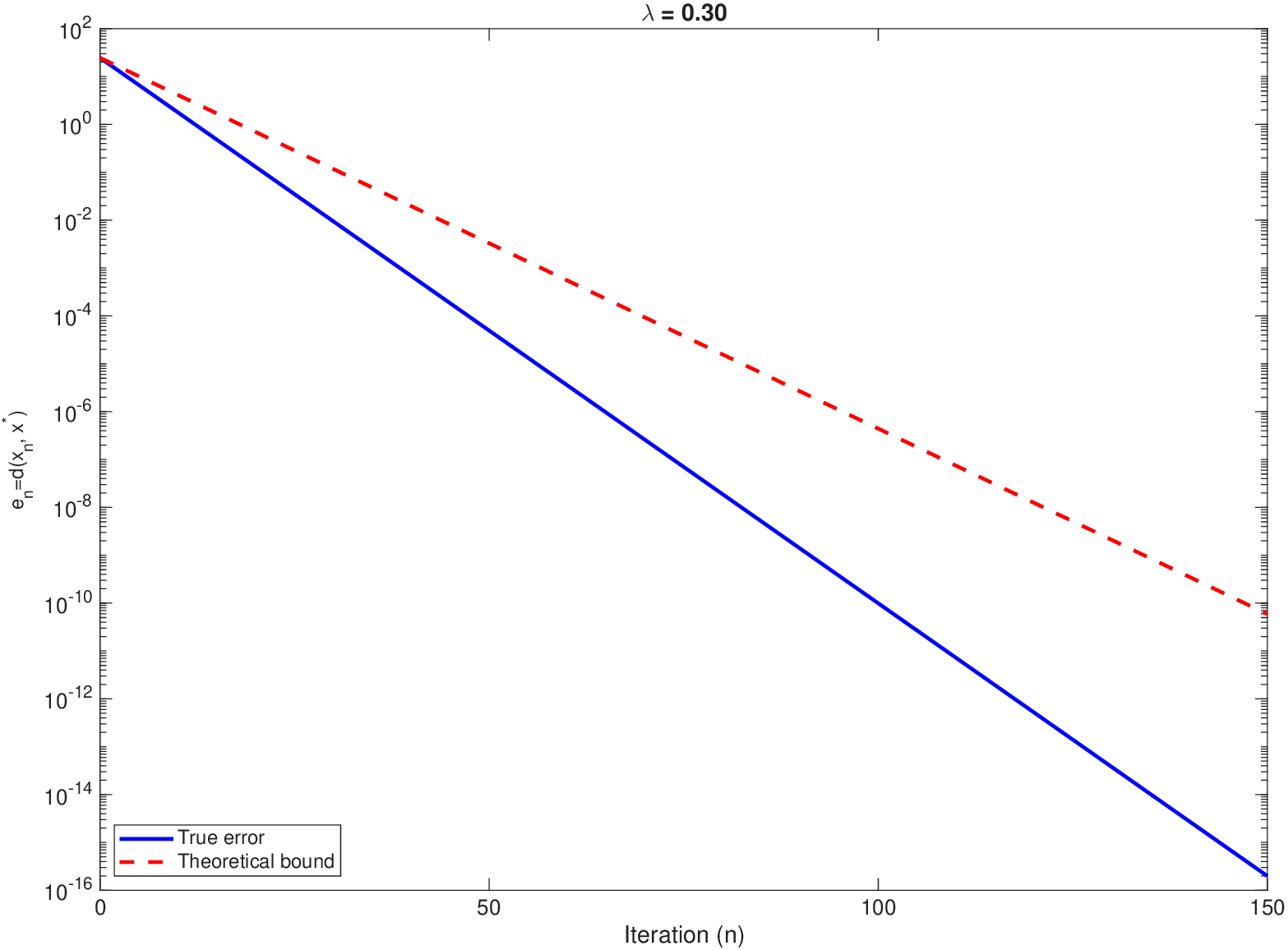}
		\caption{$\lambda=0.30$}
		\label{fig:rate-030}
	\end{subfigure}
	
		\caption{Comparison of the true error and the theoretical
			convergence bound for Alg. REMB and Alg. REMD
			in Example~\ref{Ex22} with $N=100$.}
	\label{fig:linear-convergence}
	
\end{figure}

{}{Table~\ref{tab:linear-rate-factors} and
	Figure~\ref{fig:linear-convergence} illustrate the linear
	convergence of both algorithms. The observed mean contraction
	factors agree with the exact values $q=(1+\lambda)^{-1}$
	to the displayed precision, and the true errors remain
	below the theoretical envelopes $e_0r^n$. Convergence
	becomes faster as $\lambda$ increases. For all tested
	parameter values, $q<r$, indicating that the theoretical
	rate bound is conservative for this example.}

\begin{example}\label{NCEp}
	We consider a Nash–Cournot oligopolistic equilibrium model \cite{Facchinei}  with affine price and tax-fee functions. This model gives rise to the equilibrium bifunction \(F\) defined below.

	Consider a commodity produced by $m$ different companies. Let
	\(
	x=(x_1,x_2,\ldots,x_m),
	\)
	where $x_j$ denotes the quantity of the commodity produced by the
	$j$-th company. The total production is given by
	\(
	\delta=\sum_{j=1}^{m}x_j.
	\)
	For each company $j$, the price function is assumed to be a decreasing
	affine function of the total production $\delta$, given by
	\(
	P_j(\delta)=Q_j-\varepsilon_j\delta,
	\)
	where $Q_j>0$ and $\varepsilon_j>0$.
	The profit of company $j$ is defined by
	\(
	\zeta_j(x)
	=
	P_j(\delta)x_j-c_j(x_j),
	\)
	where $c_j(x_j)$ denotes the tax-fee (or production cost) associated
	with the production level $x_j$.
	The strategy set of company $j$ is
	\(
	C_j=[x_j^{\min},x_j^{\max}] \subset \mathbb{R}_{++},
	\)
	where $x_j^{\min}$ and $x_j^{\max}$ denote the minimum and maximum
	quantities of the commodity that can be produced by the $j$-th company,
	respectively.
	Then the feasible strategy set of the Nash--Cournot model is
	\(
	C=C_1\times C_2\times\cdots\times C_m \subset \mathbb{R}_{++}^m.
	\)
	
	Assuming that the production levels of the other companies are fixed, each
	company aims to maximize its profit by choosing an appropriate production
	level. This problem can be formulated using the well-known concept of Nash
	equilibrium. More precisely, a vector $x^*\in C$ is called a Nash equilibrium if
	\[
	\zeta_j(x^*)
	\geq
	\zeta_j\bigl(x^*[x_j]\bigr)
	\text{ for all }\,x_j\in C_j,\quad
	j=1,2,\ldots,m,
	\]
	where $x^*[x_j]$ denotes the vector obtained from $x^*$ by
	replacing its $j$-th component $x_j^*$ by $x_j$.
	
	Define
	\(
	\Psi(x,y)
	=
	-\sum_{j=1}^{m}
	\zeta_j\bigl(x[y_j]\bigr),
	\)
	and the bifunction
	\(
	F(x,y)
	=
	\Psi(x,y)-\Psi(x,x).
	\)
	Then the Nash--Cournot equilibrium problem can be formulated as the
	following equilibrium problem:
	\[
	\text{Find } x^*\in C
	\text{ such that }
	F(x^*,x)\geq0 \text{ for all } x\in C.
	\]
\end{example}

{Assume that each company $j$ has an increasing affine tax-fee function
	$c_j(x_j)$.  Thus, the total tax-fee increases with production, while the marginal tax-fee remains constant.
In this case, the bifunction $F$ can be expressed as
	\begin{equation}
		F(x,y)
		=
		\left\langle Px+Qy+r,\,y-x\right\rangle,
	\end{equation}
	where $P,Q\in\mathbb{R}^{m\times m}$ are symmetric matrices and
	$r\in\mathbb{R}^{m}$.
	
	If $Q$ is positive semidefinite and $Q-P$ is negative semidefinite, then the
	bifunction $F$ satisfies Assumptions $(A1)$--$(A3)$; see \cite{Tan2024}. The boundedness of \(C\) also ensures that Assumption \((A4)\) is satisfied.
	
	%

	For the numerical experiment, we consider four companies, with the parameters for their price functions, tax-fee functions, and strategy sets chosen as in \cite{Khammahawong2020,Tan2024}, which are defined as follows:

	\begin{itemize}
		\item First company, price $P_1(\delta)=100-0.01\delta$, tax-fee 
		$c_1(x_1)=20x_1$, and strategy set of the company 
		$C_1=[500,2000]$.
		
		\item Second company, price $P_2(\delta)=110-0.02\delta$, tax-fee 
		$c_2(x_2)=15x_2+100$, and strategy set of the company 
		$C_2=[500,2500]$.
		
		\item Third company, price $P_3(\delta)=100-0.015\delta$, tax-fee 
		$c_3(x_3)=17x_3$, and strategy set of the company 
		$C_3=[500,1500]$.
		
		\item Fourth company, price $P_4(\delta)=115-0.05\delta$, tax-fee 
		$c_4(x_4)=20x_4+75$, and strategy set of the company 
		$C_4=[500,3000]$.
\end{itemize}}

Since there are four companies, we have $m=4$, and hence
\(
C=C_1\times C_2\times C_3\times C_4
\subset \mathbb{R}_{++}^{4}.
\)
To find a solution to the Nash--Cournot equilibrium problem,
we consider two initial points,
\(
x_0=(570,948,503,812)\)
and
\(x_0=(620,932,511,808).
\)
For Algorithms REMB and REMD, we use the sequence
\(
\lambda_n=0.8+\frac{1}{(n+1000)^2}.
\)
For Algorithm TQY, we set $\tau_0=0.01$, $\delta=0.1$,
and $\chi=1.5$, and choose the sequences
\(
\xi_n=1+\frac{1}{(n+1)^3},
\qquad
\sigma_n=\frac{1}{(n+1000)^2}.
\)
The optimization problems arising in all algorithms are
solved using the \texttt{fmincon} function from the
MATLAB Optimization Toolbox. Since the equilibrium
solution is known, we measure the accuracy of the
iterates using the true error
\(
\operatorname{Er}(n):=d(x_n,x^*),
\)
where
\(
x^*=\left(2000,500,3800/3,500\right).
\)
The stopping criterion for all algorithms is
\(
\operatorname{Er}(n)\leq 10^{-7}.
\)

To solve the subproblem arising in the first step of
Algorithms REMB and REMD, we employ Algorithm~2 of
\cite{Khammahawong2020}, with the stepsize sequence
$\eta_j=1/(j+1)$. To examine the sensitivity of the
algorithms to the accuracy of the inner procedure,
we consider three stopping tolerances
$\varepsilon_{\mathrm{inner}}\in
\{10^{-2},10^{-4},10^{-6}\}$, where the inner
iteration terminates when
$d(z_{j+1},z_j)\leq\varepsilon_{\mathrm{inner}}$.

%
%
%

\begin{table}[htbp]
	\centering
	
	\begin{tabular}{cccccc}
		\hline
		Inner tolerance & Algorithm & Inn. Itr. & Out. Itr.
		& Time (s) & $Er(n)$ \\
		\hline
		\multirow{3}{*}{$10^{-2}$}
		& REMB & 7 & 3 & 0.070477 & $7.4513\times10^{-9}$ \\
		& REMD & 7 & 3 & 0.060847 & $7.7513\times10^{-9}$ \\
		& TQY  & -- & 34 & 0.388069 & $4.3069\times10^{-8}$ \\
		\hline
		\multirow{3}{*}{$10^{-4}$}
		& REMB & 8 & 3 & 0.067044 & $7.4513\times10^{-9}$ \\
		& REMD & 8 & 3 & 0.059090 & $7.7513\times10^{-9}$ \\
		& TQY  & -- & 34 & 0.388069 & $4.3069\times10^{-8}$ \\
		\hline
		\multirow{3}{*}{$10^{-6}$}
		& REMB & 9 & 3 & 0.079732 & $7.4509\times10^{-9}$ \\
		& REMD & 9 & 3 & 0.066430 & $7.4509\times10^{-9}$ \\
		& TQY  & -- & 34 & 0.388069 & $4.3069\times10^{-8}$ \\
		\hline
	\end{tabular}
	\caption{Sensitivity analysis with respect to the inner
		tolerance for the Nash--Cournot equilibrium problem
		with $x_0=(570,948,503,812)$.}
	\label{tab:inner_sensitivity1}
\end{table}

\begin{table}[htbp]
	\centering
	
	\begin{tabular}{cccccc}
		\hline
		Inner tolerance & Algorithm & Inn. Itr. & Out. Itr.
		& Time (s) & $Er(n)$ \\
		\hline
		\multirow{3}{*}{$10^{-2}$}
		& REMB & 7 & 3 & 0.061648 & $7.4513\times10^{-9}$ \\
		& REMD & 7 & 3 & 0.067911 & $7.7499\times10^{-9}$ \\
		& TQY  & -- & 35 & 0.240759 & $4.0000\times10^{-8}$ \\
		\hline
		\multirow{3}{*}{$10^{-4}$}
		& REMB & 8 & 3 & 0.068420 & $7.4513\times10^{-9}$ \\
		& REMD & 8 & 3 & 0.074998 & $7.7499\times10^{-9}$ \\
		& TQY  & -- & 35 & 0.240759 & $4.0000\times10^{-8}$ \\
		\hline
		\multirow{3}{*}{$10^{-6}$}
		& REMB & 9 & 3 & 0.069004 & $7.4509\times10^{-9}$ \\
		& REMD & 9 & 3 & 0.073045 & $7.4509\times10^{-9}$ \\
		& TQY  & -- & 35 & 0.240759 & $4.0000\times10^{-8}$ \\
		\hline
	\end{tabular}
	\caption{Sensitivity analysis with respect to the inner
		tolerance for the Nash--Cournot equilibrium problem
		with $x_0=(620,932,511,808)$.}
	\label{tab:inner_sensitivity2}
\end{table}

{}{Tables~\ref{tab:inner_sensitivity1} and \ref{tab:inner_sensitivity2} report the results for the two initial points. For each initial point and all three inner tolerances, REMB and REMD terminate after three outer iterations, whereas TQY requires 34 and 35 iterations for the first and second initial points, respectively. As the inner tolerance decreases from \(10^{-2}\) to \(10^{-6}\), the total number of inner iterations increases from 7 to 9 for both REMB and REMD. Their terminal true errors remain of order \(10^{-9}\), showing little sensitivity to the tested inner tolerances in these runs.}

{}{For the first initial point, REMD has a lower reported computational time than REMB at every tested inner tolerance. For the second initial point, REMB has a lower reported computational time than REMD. Thus, the two implementations have identical reported iteration counts but different computational times.}
{}{Both REMB and REMD require fewer outer iterations and less computational time than TQY in the reported runs. 
	Figure 6 shows the corresponding true-error histories for the two initial points.
}

\begin{figure}[htbp]
	\centering
	
	\begin{subfigure}[b]{0.45\textwidth}
		\centering
		\includegraphics[width=\textwidth]{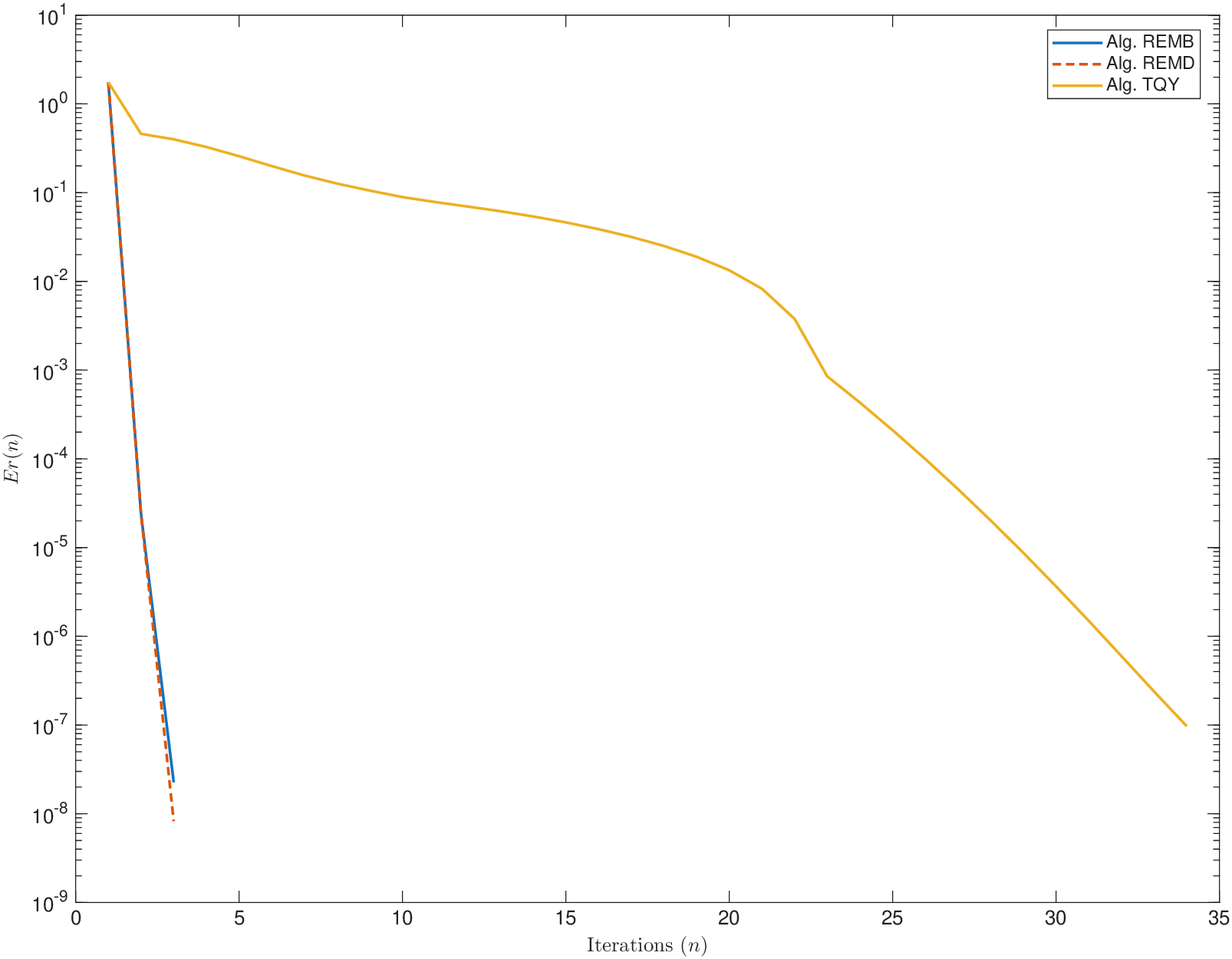}
		\caption{$x_0=(570, 948, 503, 812)$}
	\end{subfigure}
	\hfill
	\begin{subfigure}[b]{0.45\textwidth}
		\centering
		\includegraphics[width=\textwidth]{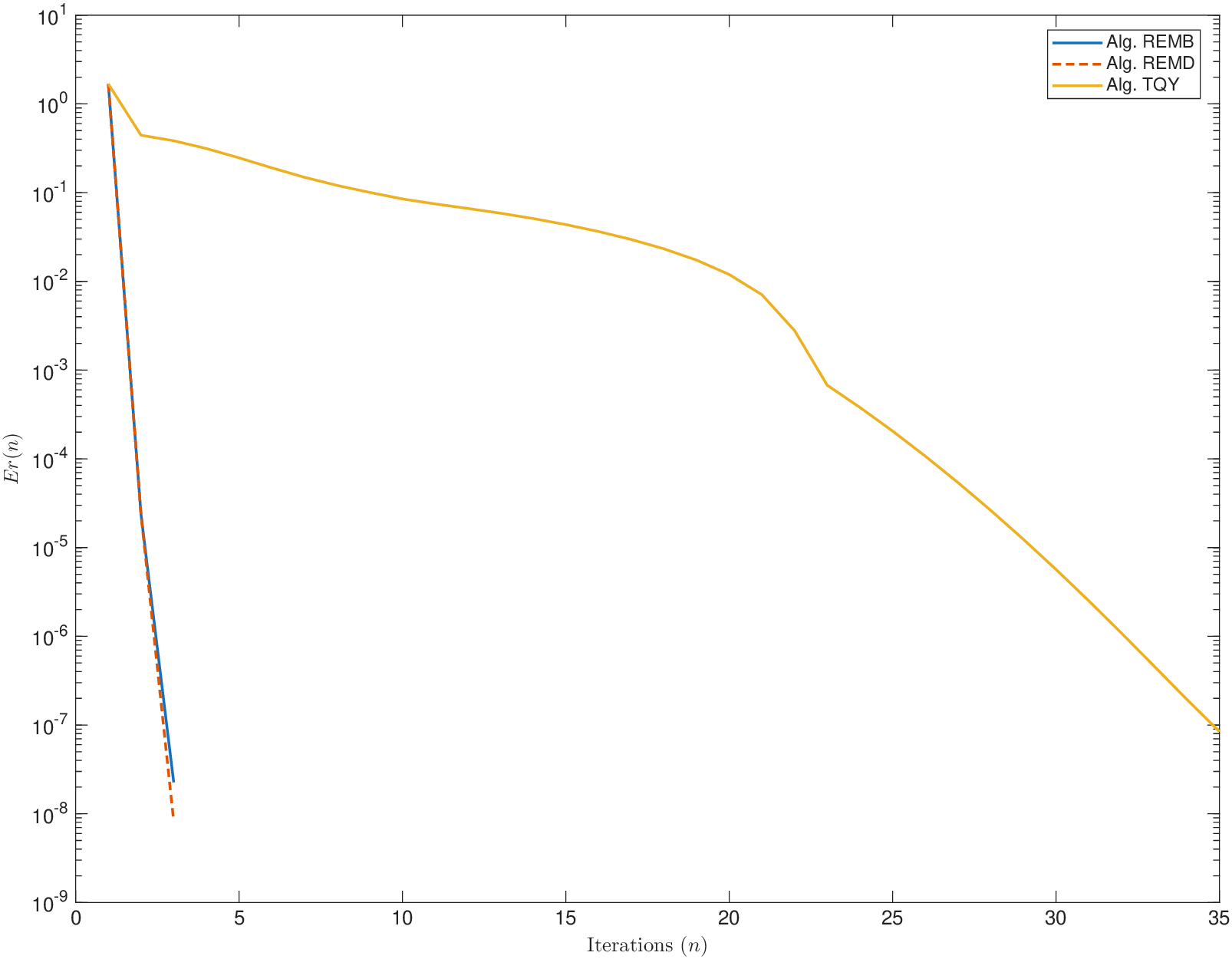}
		\caption{$x_0=(620, 932, 511, 808)$}
	\end{subfigure}
	
	\caption{Convergence behavior of Alg. TQY, Alg. REMB and Alg. REMD  for Example~\ref{NCEp}}
	\label{NCEfig}
\end{figure}

	%
	%
	%

	We next present numerical experiments on a hyperbolic space, which possesses nonzero sectional curvature, to further investigate the behavior of the proposed algorithms under a different geometric setting.

		Consider $\mathbb{E}^{N,1}$, which denotes the vector space $\mathbb{R}^{N+1}$ endowed with the symmetric bilinear form (which is called the {Lorentz metric}) that associates to vectors $x=(x_1,\ldots ,x_{N+1})$ and $y=(y_1, \ldots , y_{N+1})$, the real number $\langle x,y\rangle$ defined by
		\begin{equation*}
			\langle x , y\rangle = \sum_{i=1}^N {x_i y_i} - x_{N+1} y_{N+1}.
		\end{equation*}
		Then the  Real hyperbolic $N$-space $\mathbb{H}^N$ is the space defined by
		\begin{equation*}
			\mathbb{H}^N = \{x=(x_1,\ldots , x_{N+1}) \in \mathbb{E}^{N,1} \colon \langle x ,x \rangle =-1, x_{N+1} >0\},
		\end{equation*}
		which is the upper sheet of the hyperboloid $\{x\in \mathbb{E}^{N,1}\colon \langle x,x\rangle=-1\}$. Note that $x_{N+1}\geq 1$ for any $x\in \mathbb{H}^N$, with equality if and only if $x_i=0$ for all $i=1, \ldots,N$. The metric of $\mathbb{H}^N$ is induced by the Lorentz metric, i.e., for $x,y\in \mathbb{H}^N$,
		\begin{equation*}
			d(x,y)=\cosh^{-1}(-\langle x , y \rangle).
		\end{equation*}
		Then $\mathbb{H}^N$ is an Hadamard manifold with sectional curvature $-1$. For more details, we refer to \cite{Ferreira2005}. The normalized geodesic $\gamma \colon \mathbb{R} \to \mathbb{H}^N$ starting from $x \in \mathbb{H}^N$ is given by
		\begin{equation*}
			\gamma(t)= (\cosh(t))x+(\sinh(t))v \text{ for all } x,y \in \mathbb{H}^N \text{ and } t \in [0,1],
		\end{equation*}
		where $v \in T_x\mathbb{H}^N$ is a unit vector. The exponential map can be expressed as $\exp_x\colon T_x\mathbb{H}^N
		\to \mathbb{H}^N$ so that
		\begin{equation*}  \label{11}
			\exp_x(rv)= (\cosh r)x + (\sinh r)v,
		\end{equation*}
		for any $r \in \mathbb{R}^+$, $x\in \mathbb{H}^N$ and any unit vector $v \in
		T_x\mathbb{H}^N$. 
		The inverse of the exponential map is given by
		\begin{equation*}
			\exp_x^{-1}y = \cosh ^{-1}(- \langle x,y \rangle) \frac{y+\langle x,y
				\rangle x}{\sqrt{\langle x,y \rangle^2 -1}}.
		\end{equation*}
		For distinct $z,x\in\mathbb H^N$, the associated
		Busemann function $b_{\gamma_{z,x}} \colon \mathbb{H}^N \to  \R$ is given by 
		\[b_{\gamma_{z,x}}(y)=\ln\left(-\left\langle y,\,z+
		\frac{
			x+\langle z,x\rangle z
		}{
			\sqrt{\langle z,x\rangle^2-1}
		}
		\right\rangle
		\right) \text{ for all } y \in  \mathbb{H}^N.\]
		For more details, see \cite{diaz2026busemann}.

\begin{example}\label{Ex2}
	\rm
	{}{Let $a=(0,\ldots,0,1)\in\mathbb{H}^N$ and
	\(
	C=\{x\in\mathbb{H}^N:d(x,a)\leq R\}\), \(R>0.
	\)
	Choose distinct points  $a_1,\ldots,a_m\in C$, and set
	\[
	\phi_i(x)=\frac12d^2(x,a_i),\qquad
	\Phi(x)=\bigl(\phi_1(x),\ldots,\phi_m(x)\bigr)^{\mathsf T}.
	\]
	Let $b\in(0,\infty)^m$ and let $A\in\mathbb{R}^{m\times m}$
	be symmetric positive semidefinite with nonnegative entries.
	Define $F\colon C\times C\to\mathbb{R}$ by
	\begin{equation}\label{bifn3}
		F(x,y)=
		\bigl(b+A\Phi(x)\bigr)^{\mathsf T}
		\bigl(\Phi(y)-\Phi(x)\bigr) \text{ for all } x,y \in C.
	\end{equation}}
	Clearly, \(F(x,x)=0\) for every \(x\in C\). Moreover, since \(A\) is symmetric positive semidefinite,
	\[
	\begin{aligned}
		F(x,y)+F(y,x)
		=-\bigl(\Phi(x)-\Phi(y)\bigr)^{\mathsf T}
		A\bigl(\Phi(x)-\Phi(y)\bigr)\leq 0 \text{ for all } x,y\in C.
	\end{aligned}
	\]
	 Hence, \(F\) is monotone. For each fixed $x\in C$, the coefficients
	 $b_i+(A\Phi(x))_i$ are positive, since $b_i>0$, $A$ has
	 nonnegative entries, and $\phi_i(x)\geq 0$. As a function of
	 $y$, $F(x,y)$ is a positively weighted sum of the smooth,
	 geodesically convex functions $\phi_i(y)$, up to a constant.
	 Therefore, $F(x,\cdot)$ is smooth and geodesically convex, so
	 assumption $(A2)$ holds. Continuity of $F$ implies $(A3)$. Since $C$ is bounded,
	 no sequence in $C$ can escape to infinity. Thus,
	 $(A4)$ holds vacuously.
	
	To verify the Lipschitz-type condition, put
	\[
	L^2=\sum_{i=1}^{m}\bigl(R+d(a,a_i)\bigr)^2.
	\]
	Using Remark \ref{dist-grad-rem}, $\|\operatorname{grad}\phi_i(u)\|=d(u,a_i)
	\leq R+d(a,a_i)$ on $C$. Thus, we have
	$\|\Phi(x)-\Phi(y)\|_2\leq L d(x,y)$. Consequently,
	\[
	\begin{aligned}
		F(x,z)-F(x,y)-F(y,z)
		&=\bigl(A(\Phi(x)-\Phi(y))\bigr)^{\mathsf T}
		\bigl(\Phi(z)-\Phi(y)\bigr)\\
		&\leq \|A\|_2L^2\,d(x,y)d(y,z).
	\end{aligned}
	\]
	Therefore, $(A1)$--$(A4)$ hold, and the bifunction and
	feasible set also satisfy the assumptions required for
	Alg. TQY.

\end{example}


For the numerical experiments, we consider $\mathbb{H}^{N}$,
$N\in\{2,10,50\}$, and the bifunction $F$
defined in \eqref{bifn3}. We choose $R=2$ in the definition of the geodesically convex set $C$. We set $m=N+2$ and choose
\[
A=\frac1m\left(0.4I+\frac{0.2}{m}cc^{\mathsf T}\right),
\qquad
b_i=\frac{1+0.3i/m}{\sum_{j=1}^{m}(1+0.3j/m)},
\]
where $c=(1,\ldots,1)^{\mathsf T}\in\mathbb{R}^m$ and $I$
is the $m\times m$ identity matrix. Let $e_1,\ldots,e_N$ denote the standard basis vectors of
$\mathbb{R}^N$. We choose the points $a_1,\ldots,a_m\in\mathbb{H}^N$
as follows:
\[
a_i=\bigl(\sinh(r_i)e_i,\cosh(r_i)\bigr),
\qquad
r_i=0.4+\frac{0.6(i-1)}{N-1},
\quad i=1,\ldots,N,
\]
and \(
a_{N+1}=\bigl(-\sinh(0.7)e_1,\cosh(0.7)\bigr)\),
\(a_{N+2}=\bigl(-\sinh(1)e_2,\cosh(1)\bigr)
\).


For each dimension $N$, we generate the initial point as
\[
x_0=\left(\sinh(1.5)\frac{v}{\|v\|},\,\cosh(1.5)\right),
\]
where $v$ is generated using the MATLAB command
\texttt{v = randn(N,1)}. The same initial point is used for all
methods in each dimension.

{}{For REMB and REMD, we use
\(
\lambda_n=100+\frac{1}{(n+1000)^2}\).
For TQY, we set $\tau_0=0.01$, $\delta=0.1$, $\chi=1.5$,
$\xi_n=1+1/(n+1)^3$, and $\sigma_n=1/(n+1000)^2$.
To solve the subproblem arising in the first step of REMB
and REMD, we employ Algorithm~2 of \cite{Khammahawong2020}, with inner step sizes $\eta_k=1/(k+1)$. For each outer iteration
$n$, the inner sequence $\{z_k\}$ starts at $z_0=x_n$
and terminates when $d(z_k,z_{k-1})\leq10^{-3}$.
All minimization subproblems are solved using \texttt{fmincon}. For all
methods, the outer stopping criterion is
\(
E_n=d(x_{n+1},x_n)\leq10^{-8}.
\)}


{}{Table~\ref{Tab4} reports the iteration counts and computational times
	for REMB, REMD, and TQY in dimensions \(N=2\), \(N=10\), and \(N=50\).
	REMB and REMD require the same number of outer iterations in each
	dimension, and both require fewer outer iterations and less computational
	time than TQY in all three dimensions. This trend is also evident in Figure~\ref{Hyperbolic}. Despite their identical outer
	iteration counts, REMD requires less computational time than REMB in
	each dimension.}

	\begin{table}[htbp]
		\centering
		\adjustbox{max width=\textwidth}{
			\begin{tabular}{lccccccccc}
				\hline
				& \multicolumn{3}{c}{$N=2$} & \multicolumn{3}{c}{$N=10$} & \multicolumn{3}{c}{$N=50$} \\
				Algorithms & REMB & REMD & TQY & REMB & REMD & TQY & REMB & REMD & TQY \\
				\hline
				Itr. $(n)$ & 24 & 24 & 303 & 26 & 26 & 321 & 26 & 26 & 323 \\
				Time (s) & 0.178082 & 0.107327 & 0.279058 & 0.670967 & 0.530767 & 0.686046 & 8.413071 & 6.271180 & 10.648644 \\
				\hline
		\end{tabular}}
		\caption{Numerical results for the three algorithms in dimensions $N=2,10,50$.}
		\label{Tab4}
	\end{table}

	\begin{figure}[htbp]
		\centering
		
		\begin{subfigure}[b]{0.45\textwidth}
			\centering
			\includegraphics[width=\textwidth]{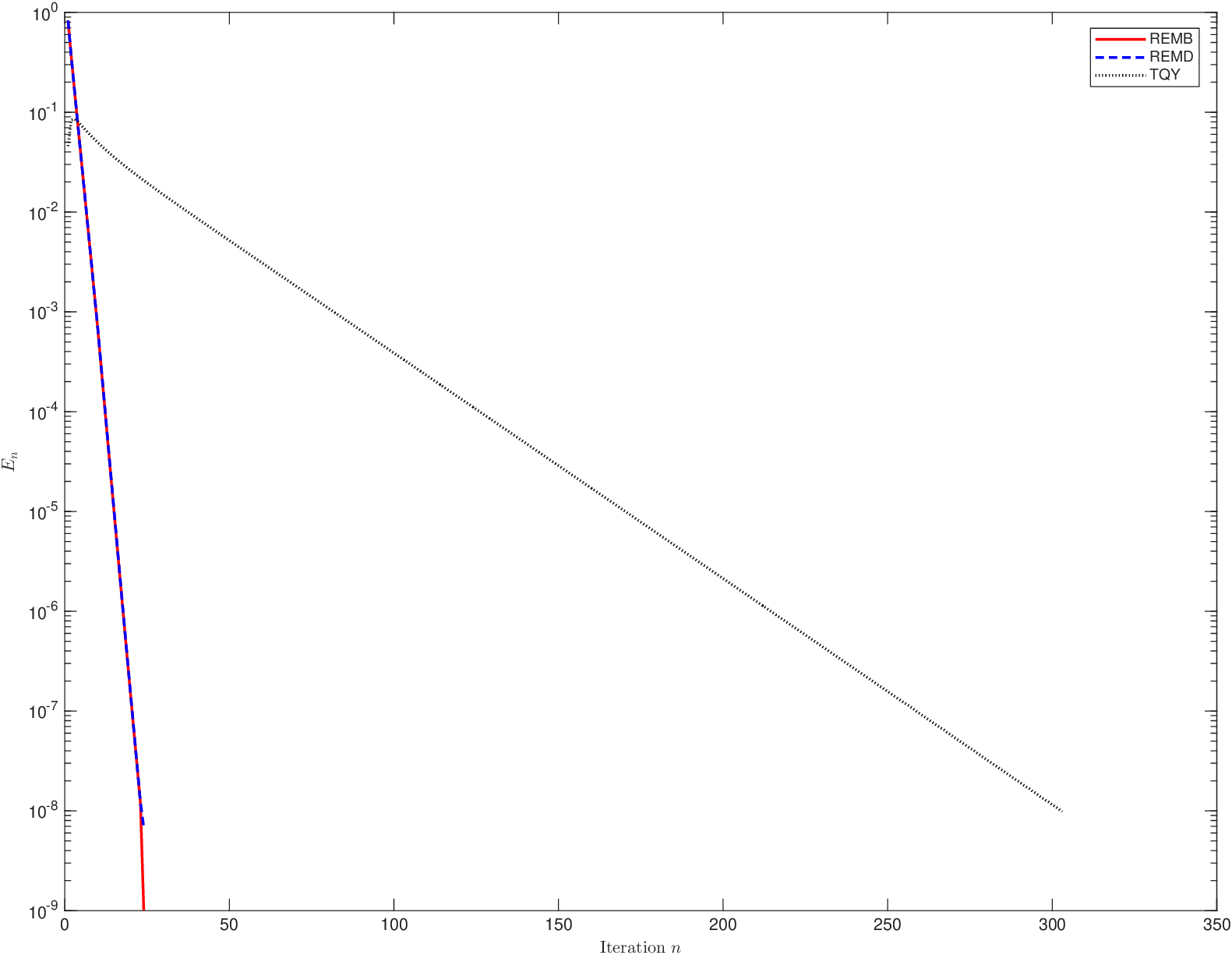}
			\caption{$N=2$}
		\end{subfigure}
		\hfill
		\begin{subfigure}[b]{0.45\textwidth}
			\centering
			\includegraphics[width=\textwidth]{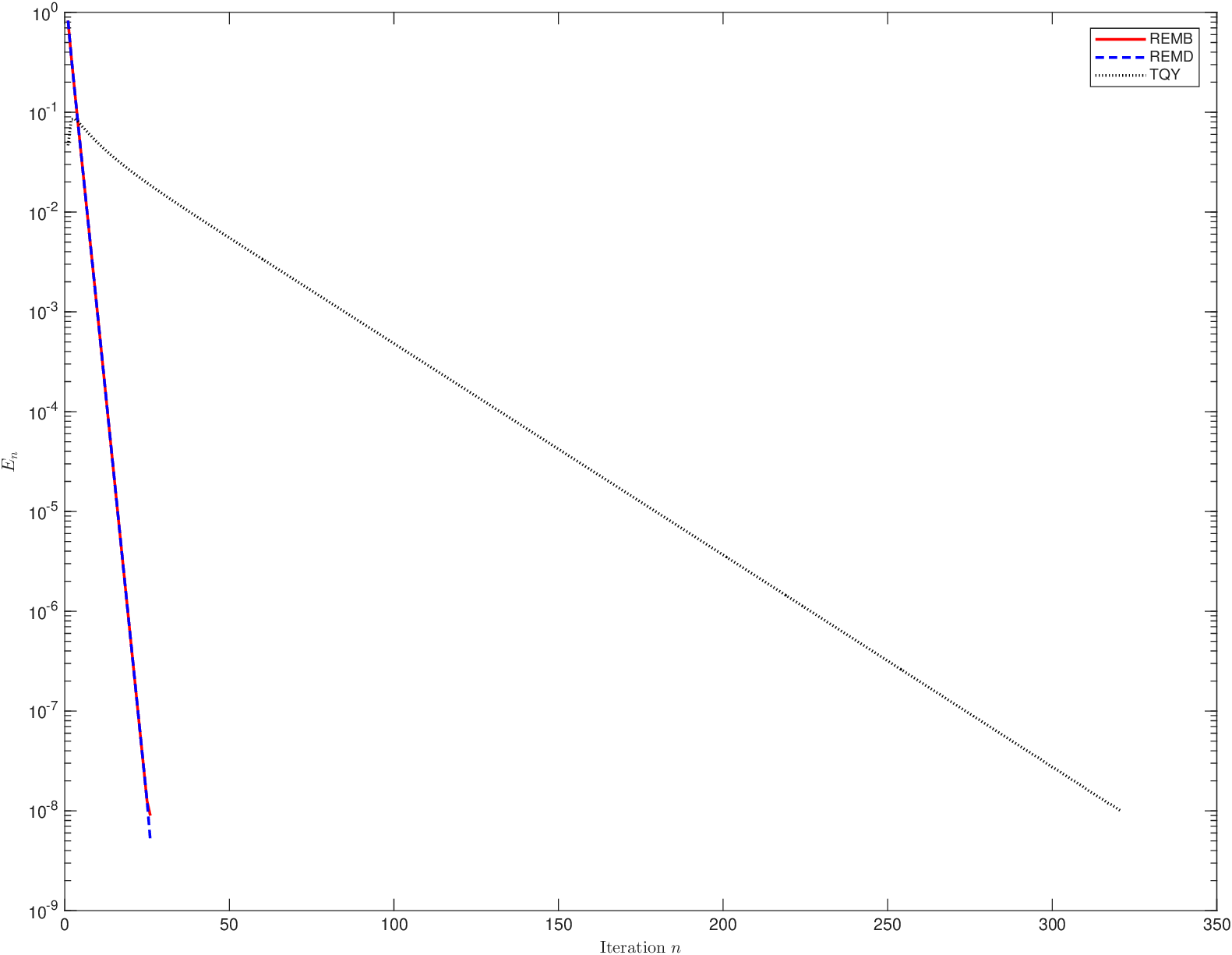}
			\caption{$N=10$}
		\end{subfigure}
		\hfill
		\begin{subfigure}[b]{0.45\textwidth}
			\centering
			\includegraphics[width=\textwidth]{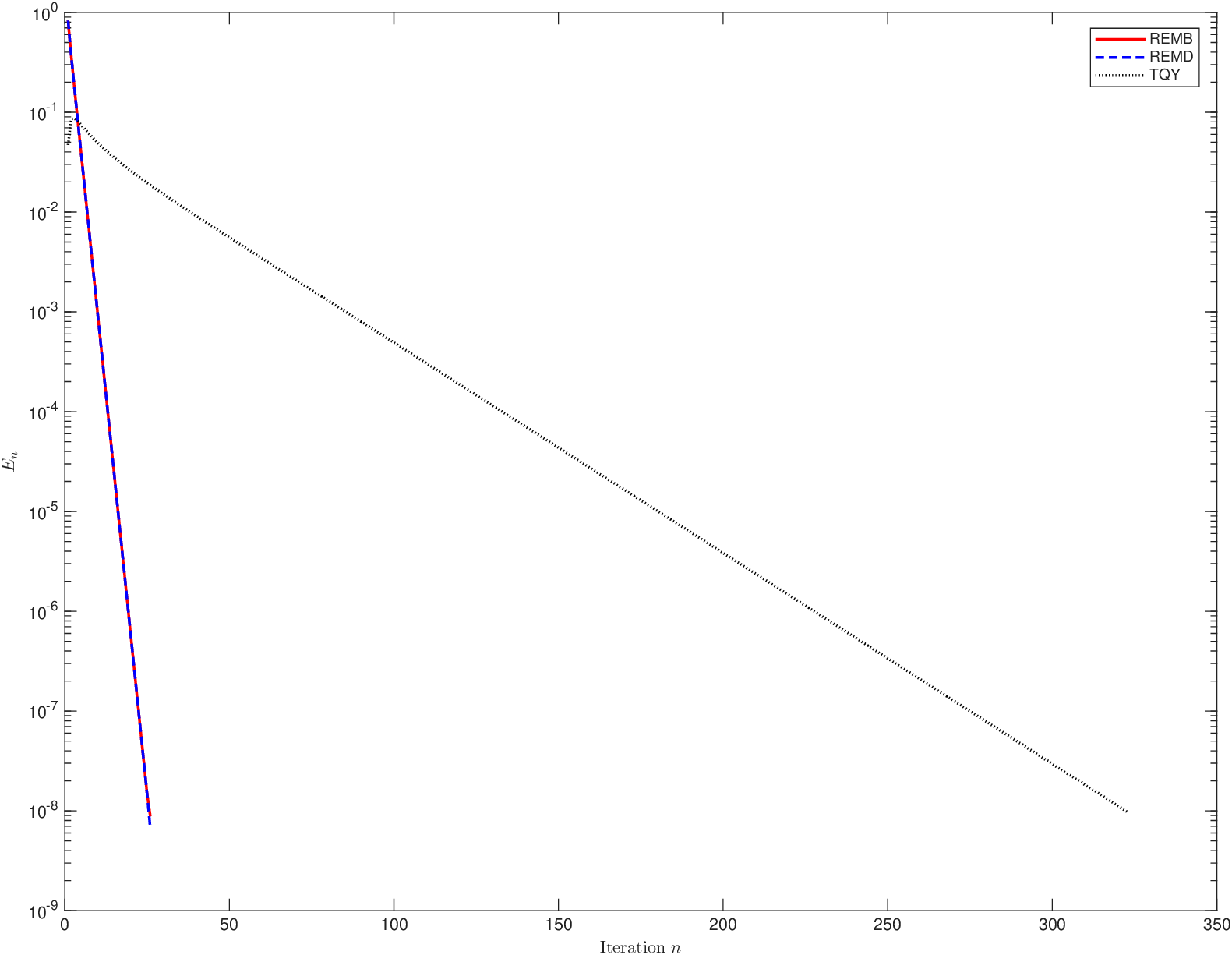}
			\caption{$N=50$}
		\end{subfigure}
	
		\caption{Convergence behavior of Alg. TQY, Alg. REMB and Alg. REMD  for N=2, 10, 50 for Example~\ref{Ex2}}
		\label{Hyperbolic}
	\end{figure}

%
%
%
%
	

	\section{Conclusion}

	In this work we introduced two regularized extragradient methods for solving equilibrium problems on Hadamard manifolds using two different regularization strategies based on the Busemann function and a distance-type metric function. {}{Under the stated assumptions, the two
	formulations generate identical exact iterates when
	initialized at the same point and supplied with the same
	parameter sequence.} We established convergence under
	monotonicity without requiring a Lipschitz-type condition
	on the bifunction, {}{provided that the positive regularization
	parameters are bounded away from zero}. In addition, we derived global error bounds and proved that the generated sequences converge at an \(R\)-linear rate when the bifunction is also strongly pseudomonotone. Moreover, we established a lower bound for the Busemann function on Hadamard manifolds in terms of the inverse exponential map.
	
	The numerical experiments compared inexact implementations
	of the two regularization formulations. {}{Their relative
	computational performance varied across the tested
	problems, reflecting differences in the formulations
	and their numerical solution procedures.} We also compared the proposed methods with existing algorithms from the literature, and the results indicate that the proposed methods are competitive and effective in terms of convergence behavior and computational performance.  Finally,
	the application to a Nash--Cournot equilibrium problem
	illustrated their use in an equilibrium model arising
	in economics.
	
	

	\bmhead{Acknowledgments}
	                 
	The authors are very grateful to the anonymous referees for their close reading, pertinent comments and valuable suggestions, which have significantly improved the quality and presentation of the manuscript.
	
	\section*{Declarations}
	
	\begin{itemize}
		\item\textbf{Funding:} No funding or external support was received for this research.
		\item \textbf{Conflict of interest/Competing interests:} The authors declare that they have no competing interests.
		\item \textbf{Ethics approval and consent to participate:} Not applicable.
		\item \textbf{Consent for publication:} Not applicable.
		\item \textbf{Data availability:} Not applicable.
		\item \textbf{Materials availability:} Not applicable.
		\item \textbf{Code availability:} The MATLAB codes used for the numerical experiments are publicly available in the GitHub repository:
		\url{https://github.com/shikhers04/Regularized-Extragradient-Hadamard}.
		\item \textbf{Author contributions:} All the authors read and approved the final manuscript.
	\end{itemize}

	\FloatBarrier
	\bibliography{EP}

\end{document}